\documentclass[12pt]{amsart}
\usepackage{amsmath}
\usepackage{amssymb}
\usepackage{color}
\usepackage{ifthen}
\usepackage{amscd}
\usepackage{amsfonts}

\newtheorem{propo}{Proposition}[section]
\newtheorem{corol}[propo]{Corollary}
\newtheorem{theor}[propo]{Theorem}
\newtheorem{lemma}[propo]{Lemma}
\theoremstyle{definition}
\newtheorem{defin}[propo]{Definition}

\theoremstyle{remark}
\newtheorem{remar}[propo]{Remark}

\numberwithin{equation}{section}

\newcommand{\ex }{\mathrm{exp}}
\newcommand{\lo }{\mathrm{log}}
\newcommand{\arc }{\mathrm{arctan}}

\title[Quantum affine superalgebra $U_{v}(A(0,2)^{(4)})$]{On the structure of quantum affine superalgebra $U_{v}(A(0,2)^{(4)})$}

\author{Fengchang Li}
\address{Department of Mathematics, Nanjing University, Nanjing 210093, China.}
\address{Philipps-Universit\"at Marburg,
	Fachbereich Mathematik und Informatik,
	Hans-Meerwein-Strasse,
	D-35032 Marburg, Germany.}
\email{dg1921005@smail.nju.edu.cn}

\thanks{The work  was supported by the Chinese Scholarship Council}

\begin{document}

\begin{abstract}
We research $U_{v}(A(0,2)^{(4)})^{+}$ defined by quantum Serre relations, when $v$ is not a root of unity. We prove that $U_{v}(A(0,2)^{(4)})^{+}$ is isomorphic to a Nichols algebra. In other words, it is equivalent to define $U_{v}(A(0,2)^{(4)})^{+}$ by quantum Serre relations and by the radical of the bilinear form. We determine all the root multiplicities and give a PBW basis of $U_{v}(A(0,2)^{(4)})^{+}$. 
\end{abstract}


\maketitle

\section*{Introduction}

The theory of quantum affine (super)algebras and their representations is very rich and has important applications in many branches of mathematics and mathematical physics. One of the most fundamental question is giving a PBW basis and the relations between corresponding PBW generators. For example, for untwisted quantum affine algebras, we have the  famous Drinfeld's second realization, see \cite{Drinfeld} and \cite{beck}. For twisted case see \cite{twistedaffine}. For quantum affine superalgebras, the corresponding work appeared first in \cite{yamane1}, which is foundational in this theory. Then e.g. \cite{HSTY08}. However in super case, there are still many unknown objects, even if in rank 2 case. By van de Leur's classification \cite{vandeleur}, there is a family $A(m,n)^{(t)}$, $t=1,2,4$ in affine Kac-Moody superalgebras. The Serre-type defining relations were determined also in \cite{yamane1}. Note that when $m=n$, the cardinal of these Serre-type defining relations is infinite (also in quantum case) because of the one dimensional center spanned by the identity matrix. In this paper we talk about $A(2m,2n)^{(4)}$. The main difficulty or feature in this family is there exist odd imaginary roots and imaginary root vectors do not necessarily commute with each other. In quantum case this difficulty still exists and is even more pronounced. There is some information on the root multiplicities. In \cite{EG}, when $v$ is transcendental, the root multiplicities of $U_v(A(2m,2n)^{(4)})$ are known, which coincide with $A(2m,2n)^{(4)}$. In general the structure of these quantum affine superalgebras is a long time unsolved problem due to the non-commutativity of the imaginary root vectors. We will consider the rank 2 case in this family, i.e. $U_{v}(A(0,2)^{(4)})^{+}$, which is defined by quantum Serre relations. One can also define $U_{v}(A(0,2)^{(4)})^{+}$ by the radical of the bilinear form, see \cite{yamane1}, for non-super case, see \cite{lusztigbook}. We will prove these two definitions are equivalent when $v$ is not a root of unity.

The theory of Nichols algebras play an important role in the classification of finite-dimensional pointed Hopf algebras, see \cite{classificationpointed}. It is well-known that the theory of Nichols algebras is closely related to the theory of quantum groups. Concretely in finite case, the positive part of a quantum group is a Nichols algebra when $v$ is not a root of unity, see Chapter 16 of \cite{heckenbergerbook}. In affine case, this also holds by \cite{heckenbergerbgg}.

In this paper we will consider first the Nichols algebra $B(\mathbb{V})$ of braided vector space $\mathbb{V}$,   whose braiding matrix is 
$
\begin{pmatrix}
    q & q^{-1} \\
    q^{-1} & -q 
\end{pmatrix}
$. This Nichols algebra first appeared in Cuntz's classification of rank 2 affine Nichols algebras \cite{cuntz}.  All the  Nichols algebras in this classification are quite open except the positive part of $U_{v}(\hat{sl_2})^{+}$, which has been very well-known. We believe these Nichols algebras have Lie theory background, especially some of them should come from the positive part of quantum (super)group.

The Dynkin diagrams of $U_{v}(\hat{sl_2})$ and $U_{v}(A(0,2)^{(4)})$ are almost the same, the only difference is $U_{v}(A(0,2)^{(4)})$ has an odd point. The corresponding Nichols algebra of $U_{v}(\hat{sl_2})^{+}$ has braiding matrix $
\begin{pmatrix}
    q & q^{-1} \\
    q^{-1} & q 
\end{pmatrix}
$, the only difference between this and the braiding matrix of our $B(\mathbb{V})$ is a negative sign. They share some common properties but are also very different, $B(\mathbb{V})$ is much more complicated. In this paper we will determine all the root multiplicities, give a PBW basis and prove that $B(\mathbb{V})$ is isomorphic to $U_{v}(A(0,2)^{(4)})^{+}$ when $v$ is not a root of unity.

Our main tool is Lyndon word theory, see \cite{karchenko}. We follow the notations in \cite{rootmul}. There are many important applications of Lyndon word theory. For example Heckenberger's celebrated work about Nichols algebras of diagonal type, see \cite{heckenbergerweylgroupoid}. In \cite{lyndonapp}, Angiono obtained a presentation by generators and relations of all Nichols algebras of diagonal type with finite root system.  We also use an important subquotient appeared in \cite{heckenbergersubquotient} and \cite{subquotient}.  In \cite{subquotient}, \cite{usesubquotient'}, \cite{usesubquotient},  the subquotient played a central role in the proof of the conjecture in \cite{AAH21} about the equivalence between finiteness of root system and finiteness of GK dimension. 

Moreover, in Drinfeld second realization, one of the most important technique is using exponential to describe distinguished imaginary root vectors. In this paper, exponential is also used and very beneficial. In addition, arctangent  will be introduced and play an important role.

The paper is organized as follows. In Section 1, we give some notations and introduce some basic theories. In Section 2, we derive the relations we need in $B(\mathbb{V})$ and find the central elements in the subquotient $K_{\ge 1}/K_{>1}$. With these information we can construct new distinguished imaginary root vector candidates using logarithm and arctangent in Section 3, which are primitive in the subquotient $K_{\ge 1}/K_{>1}$. Then we determine the multiplicities completely by the primitivity. In Section 4 we prove that all the commuting relations between root vectors of $B(\mathbb{V})$ can be derived from quantum Serre relations, which implies $B(\mathbb{V})$ and $U_{v}(A(0,2)^{(4)})^{+}$ are isomorphic.

\section{Preliminaries}
\label{sec:prelims}

In the whole paper, $q\in \mathbb{C}$ and is not a root of unity. Let $\theta=q-q^{-1}$, $[n]_{q}=\frac{q^n-q^{-n}}{q-q^{-1}}$.

All the brackets in this paper are braided brackets introduced in Section 1.1.  
\subsection{Nichols algebras of diagonal type} We refer to \cite{diagonaltype} for the definition of  a braided vector space of diagonal type and the definition of a Nichols algebra. 

Let $\mathcal{K}$ be a field. Let $s \in \mathbb{N}$, $X=\{x_1,...,x_s\}$. Let $\mathbb{X}$ be the set of words with letters in $X$.  Let $V$ be a $s$-dimensional vector space with basis $\{x_1,...,x_s\}$. We can identify $\mathcal{K}\mathbb{X}$ with $T(V)$. Then $T(V)$ is a $\mathbb{Z}^s$-graded algebra determined by deg $x_i=\alpha_i, 1\le i \le s$, where $\{\alpha_1,...,\alpha_s\}$ is the canonical basis of $\mathbb{Z}^s$.

Let $c$ be a braiding of $V$. The braided bracket of $x,y\in T(V)$ is defined by $$[x,y]_c=m\circ (id-c)(x\otimes y).$$

Assume that $(V,c)$ is of diagonal type with braiding matrix $(q_{ij})_{1\le i,i\le s}$ with respect to the basis $\{x_1,...,x_s\}$, and let $\chi:\mathbb{Z}^{s}\times \mathbb{Z}^{s}\rightarrow \mathcal{K}^{\times}$ be the bicharacter determined by the condition $\chi(\alpha_i,\alpha_j)=q_{ij}$ for each pair $1\le i,j\le s$. Then for each pair of $\mathbb{Z}^{s}$-homogeneous elements $u,v\in \mathbb{X}$, $c(u\otimes v)=q_{u,v}v\otimes u$, $q_{u,v}=\chi($deg $u,$deg $v)\in \mathcal{K}^{\times}$. In such case, the ``braided Jacobi identity" and skew derivations hold:
\begin{align}
[[u,v],w]&=[u,[v,w]]-\chi(\alpha,\beta)v[u,w]+\chi(\beta,\gamma)[u,w]v,\nonumber \\
[u,vw]&=[u,v]w+\chi(\alpha,\beta)v[u,w],\nonumber \\
[uv,w]&=\chi(\beta,\gamma)[u,w]v+u[v,w],\nonumber 
\end{align}
where $u$, $v$, $w$ are homogeneous elements in $T(V)$, $\alpha$, $\beta$, $\gamma$ are the degrees of $u$, $v$, $w$. If $\chi(\beta,\alpha)\chi(\alpha,\beta)=1$, then we have 

\begin{align}
[[u,v],w]=[u,[v,w]]-\chi(\alpha,\beta)[v,[u,w]].
\end{align}

\subsection{Braided vector space $\mathbb{V}$, the Nichols algebra $B(\mathbb{V})$ and the root system of $B(\mathbb{V})$}

Let $\mathbb{V}$ be a 2-dimensional braided vector space of diagonal type with braiding matrix $\begin{pmatrix}
    q & q^{-1} \\
    q^{-1} & -q 
\end{pmatrix}$, $q$ is not a root of unity. $B(\mathbb{V})$ is the main research object in this paper. We will prove $B(\mathbb{V})$ is isomorphic to $U_{v}(A(0,2)^{(4)})^{+}$ with $q=v^2$.

It is easy to see $B(\mathbb{V})$ is of affine Cartan type with Cartan matrix 
$
\begin{pmatrix}
    2 & -2 \\
    -2 & 2 
\end{pmatrix}$. Then the real roots of $B(\mathbb{V})$ are just the positive real roots of $\hat{sl_2}$, namely $(n+1)\alpha_1+n\alpha_2$ and $n\alpha_1+(n+1)\alpha_2$ for $n\ge0$. $B(\mathbb{V})$ also has imaginary roots. The only possible imaginary roots are $n\delta$, $\delta=\alpha_1+\alpha_2$ and $n\ge1$ but  the multiplicities are in general unknown.

Note that if $\beta$ is an imaginary root, then we have $\chi(\beta,\alpha)\chi(\alpha,\beta)=1$ for all $\alpha\in \mathbb{N}_0^2$, and hence (1.1) holds.

\begin{remar}
For $U_v(\hat{sl_2})^{+}$ when $v$ is not a root of unity, the corresponding Nichols algebra is $B(V)$, where $V$ has braiding matrix $\begin{pmatrix}
    q & q^{-1} \\
    q^{-1} & q 
\end{pmatrix}$ and $q=v^2$.
\end{remar}

\subsection{Differential operators}

We refer to Chapter 7 of \cite{heckenbergerbook}. 

\begin{defin}Define right and left differential operators $\partial_i^R,\partial_i^L: B(\mathbb{V})\rightarrow B(\mathbb{V}),1\le i\le 2$ by:

(1) $\partial_i^R(1)=\partial_i^L(1)=0$, $\partial_i^R(x_j)=\partial_i^L(x_j)=\delta_{ij}$ for all $1\le i,j\le2$.

(2) $\partial_i^R(xy)=x\partial_i^R(y)+\partial_i^R(x)\chi(\alpha_i,\beta)y$ for all $1\le i,j\le 2$ and $x,y\in B(\mathbb{V})$, $y$ is $\mathbb{N}_0^2$-homogeneous with degree $\beta$.  

(3) $\partial_i^L(xy)=\partial_i^L(x)y+\chi(\alpha,\alpha_i)x\partial_i^L(y)$ for all $1\le i,j\le 2$ and $x,y\in B(\mathbb{V})$, $x$ is $\mathbb{N}_0^2$-homogeneous with degree $\alpha$.
\end{defin}
\begin{propo}
For $x\in B(\mathbb{V})$ and $x$ is not a constant, in $B(\mathbb{V})$ we have 
\begin{align}
x=0   &\Leftrightarrow \partial_i^R(x)=0\ \forall 1\le i\le 2,\nonumber \\
                      &\Leftrightarrow \partial_i^L(x)=0\ \forall 1\le i\le 2.\nonumber
\end{align}
\end{propo}

\subsection{An important subquotient}

We follow the notation in \cite{subquotient}. For our braided vector space $\mathbb{V}$, we know that $B(\mathbb{V})$ has a unique $\mathbb{N}_0^2$-grading as a braided Hopf algebra $$B(\mathbb{V})=\mathop{\bigoplus}\limits_{\alpha\in \mathbb{N}_0^2}B^{\alpha}(\mathbb{V}) $$ such that deg $x_1=\alpha_1$, deg $ x_2=\alpha_2$. Now set 

\begin{align*}
B_{\geq 1} &= \bigoplus_{\substack{\alpha = a_1\alpha_1+a_2\alpha_2 \in  \mathbb{N} _0^2:\\  a_1\ge a_2}} B^{\alpha} (\mathbb{V}), &
B_{> 1} &= \bigoplus_{\substack{\alpha = a_1\alpha_1+a_2\alpha_2 \in  \mathbb{N} _0^2:\\ a_1 > a_2}} B^{\alpha} (\mathbb{V}),
\\
K_{\geq 1}&=\{x\in B (\mathbb{V})\,|\, \Delta (x)\in B_{\geq 1}\otimes B (\mathbb{V})\},&
K_{>1}&=K_{\geq 1}\cap B_{>1}.
\end{align*}

The subquotient we use is $K_{\ge1}/K_{>1}$. 
\begin{propo}
$K_{\ge1}/K_{>1}$ is a braided Hopf algebra induced by the braided Hopf algebra structure of $B(\mathbb{V})$.
\end{propo}
\begin{proof}
See Prop. 3.9 in \cite{subquotient}.

\end{proof}

\subsection{Lyndon word theory}

We follow the form of Lyndon word theory in \cite{rootmul}, for more details and a more general case see \cite{heckenbergersubquotient}. 

Follow the notations in Section 1.1. Define an ordering on $X$ by $x_1<x_2<...<x_s$, consider the lexicographic ordering on $\mathbb{X}$. Let $\mathbb{X}^{\times}$ be the set of nonempty words in $\mathbb{X}$. A word $u\in \mathbb{X}^{\times}$ is a Lyndon word if for any decomposition $u=vw$, $u,v\in \mathbb{X}^{\times}$, we have $u<w$.

A word $u\in \mathbb{X}^{\times}$ is a Lyndon word if and only if either $u\in X$, or there exist Lyndon words $v,w$ such that $v<w$ and $u=vw$.

For any word $u$ of length at least two has a unique decomposition into the product of two Lyndon words $u=vw$ where the length of $v$ is minimal. This decomposition is called Shirshov decomposition.

Recall that we identify $\mathcal{K}\mathbb{X}$ with $T(V)$. For a Lyndon word $u$, we denote $[u]\in T(V)$ inductively as follows:\\
(1) $[u]=u$, if $u\in X$,\\
(2) $[u]=[[v],[w]]$ where $u=vw$ is the Shirshov decomposition, if $|u|\ge 2$. $[,]$ is the braided bracket defined in Section 1.1.

We call $[u]$ super-letters. Then the total ordering on $\mathbb{X}$ induces a total ordering on the set of super-letters.

Let $\mathbb{L}$ be the set of Lyndon words. If $w=u_1 u_2 \cdots u_k$ for $u_1\ge u_2\ge \cdots \ge u_k$ and $u_1,...,u_k \in \mathbb{L}$ for some $k\in \mathbb{N}$, define $[w]$ by $[u_1] [u_2] \cdots [u_k]$. Denote $M_{> u}$ by the vector space spanned by all $[w]\in T(V)$ such that $w=v_1 v_2 \cdots v_k$, $v_1\ge v_2\ge \cdots \ge v_k > u$, $v_1,...,v_k \in \mathbb{L}$ for some $k\in \mathbb{N}$. 

\begin{propo}\label{comultiplicationlyndon}
Let $u\in \mathbb{L}$. Then $\Delta([u])-1\otimes [u]-[u]\otimes 1\in T(V)\otimes M_{> u}$.
\end{propo}
\begin{proof}
We refer to the proof of Prop. 3.6 in \cite{heckenbergersubquotient}.
\end{proof}

For $\alpha \in \mathbb{Z}^s$, let $o_{\alpha}\in \mathbb{N}\cup \infty$ be the order of $\chi(\alpha,\alpha)$. Let $O_{\alpha}=\{1,o_{\alpha},\infty \}$.

\begin{defin}
Let $w\in \mathbb{X}^{\times}$. We say that $[w]$ is a root vector candidate if $w=v^k$ for some Lyndon word $v$ and $k\in O_{deg\ v}\backslash \{\infty\}$.
\end{defin}

\begin{defin}
A root vector candidate $[w]$ is called a root vector of $B(V)$ if $[w]\in B(V)$ is not a linear combination of elements of the form $[v_k]^{m_k}\cdots[v_1]^{m_1}$, where $k, m_k,...,m_1 \in \mathbb{N}$ and $[v_1],...,[v_k]$ are root vector candidates or equivalently root vectors with $v_k>\cdots >v_1>w$.
\end{defin}

\begin{theor}
Let $\textbf{R} \subseteq \mathbb{X}^{\times}$ such that $w\in \textbf{R}$ if and only if $[w]$ is a root vector. Then the elements
\begin{align}
[v_k]^{m_k}\cdots[v_1]^{m_1},\ & k\in \mathbb{N},\  v_1,...,v_k\in \textbf{R},\  v_1<v_2<\cdots <v_k, \nonumber\\
&\forall 1\le i\le k,\ 0<m_i<o_{deg\ v_i},
\end{align}
form a vector space basis of $B(V)$.

\end{theor}

Let $X_n$, $Y_n$ be the real root vector candidates corresponding to Lyndon word $x_1(x_1 x_2)^{n-1}$ and $(x_1 x_2)^{n-1}x_2$. Define $L_n$ by $[X_1,Y_n]$, $L_n'$ by $[X_2,Y_{n-1}]$ (it is convenient to define $L_1=L_1'=[X_1,Y_1]$ and $L_0=L_0'=\frac{1}{\theta}$). 

\begin{remar}
We can also define $L_n$ in the same way in $U_q(\hat{sl_2})^{+}$, then $L_n$ are essentially the same as $\psi_k$ in \cite{quantumaffinealgebras}. We have $[L_m, L_n]=0$ in $U_q(\hat{sl_2})$ by Drinfeld second realization. However, we will see in our case $L_m$ and $L_n$ do not necessarily commute. This is the main difficulty in this paper.
\end{remar}

We have the following order in terms of super-letters
\begin{align}
X_1&<X_2<\cdots<X_{n}<\cdots \nonumber \\
&<\cdots<L_n<L_{n-1}<L_{n-2}<\cdots<L_2 \nonumber \\
&<\cdots<M_{2n-1}<M_{2n-3}<\cdots<M_5<M_3<M_1 \nonumber \\
&<\cdots<Y_n<\cdots<Y_2<Y_1. \nonumber 
\end{align}

\begin{lemma}\label{lemma:realrootvectors}
If $X_n$ and $Y_n$ are not 0 in $B(\mathbb{V})$, then they are root vectors at corresponding degrees.
\end{lemma}

\begin{proof}
Note that $X_n$ and $Y_n$ are the biggest super-letters at corresponding degrees. Then it follows from the definition of root vectors and degree reason.
\end{proof}

\begin{lemma} \label{subquotientlemma}
In $B(\mathbb{V})$ we have 
$$\Delta(Y_n)\equiv Y_n\otimes 1+1\otimes Y_n\ \ mod\ B_{\ge1}\otimes B(\mathbb{V})$$ and we have $X_n\in K_{\ge1}$, $L_n\in K_{\ge1}$, $Y_n\notin K_{\ge1}$, $X_n\in K_{>1}$, $L_n\notin K_{>1}$.

\end{lemma}
\begin{proof}
$\Delta(Y_n)$ follows from Prop. ~\ref{comultiplicationlyndon}. Since $L_n=[X_1,Y_n]$, then by the expression of $\Delta(Y_n)$, we have  $L_n\in K_{\ge1}$. In particular $L_1\in K_{\ge1}$, then immediately $X_n\in K_{\ge1}$. Other conclusions are obvious.
\end{proof}

Define $\hat{L}_{n+1}=[L_n, L_1]$, $\tilde{L}_n=\frac{1}{2}(L_n+L_n^{\prime})$ (It is convenient to define  $\tilde{L}_n=0$ for $n<0$). Define $M_{1}=L_1, M_3=[L_2,L_1], M_5=[L_2,[L_2,L_1]], M_7=[L_2,[L_2,[L_2,L_1]]] \cdots$

\section{Relations in $B(\mathbb{V})$}

\begin{lemma} \label{lemma1} $B(\mathbb{V})$ has the following basic properties,

(a) $[X_n,L_1^2]=X_{n+2}$ and $[L_1^2,Y_n]=Y_{n+2}$.   

(b) $L_{2n+1}^{\prime}=L_{2n+1}+[L_{2n},L_1]$, $L_{2n+2}^{\prime}=L_{2n+2}-[L_{2n+1},L_1]$.

(c) Fix n, if $\forall  i\le n$, $[L_{i-2},L_1^2]=0$, then $[X_{2k+1},Y_{n-2k}]=L_n$, $\forall 0\le k\le \frac{n-1}{2}$, $k \in \mathbb{N}_0$ and $[X_{2k+2},Y_{n-(2k+1)}]=L_n^{\prime}$, $\forall 0\le k\le \frac{n-2}{2}$, $k \in \mathbb{N}_0$.

(d) $[M_{2n+1},M_{2n-1}]=[L_2,M_{2n-1}^2]$ for all $n\ge1$.  If $[L_2,M_{2n-1}^2]=0$, then $[M_{2n+3},M_{2n-1}]=-2M_{2n+1}^2$.

\end{lemma}

We will use this lemma frequently and will not refer to it.

\begin{proof} It follows from the braided Jacobi identiy  and skew derivations and definitions directly.
\end{proof}

\subsection{Relations in root space with degree $\le 4 \delta$}
We have known quantum Serre relations holds in $B(\mathbb{V})$: \ $[X_1, X_2]=[Y_2, Y_1]=0$. In degree $\le 4 \delta$, we have the following relations. These relations are prepared for getting the central elements in next subsection and as the starting of the induction proof of Prop.~\ref{propo:realandim}.

\begin{propo}
In $B(\mathbb{V})$, the following relations hold.  \label{propo1}
\begin{subequations}
\begin{align}
[X_{1},L_{2}]&=(q+2)X_{3}-\theta L_{1}X_{2}\\
[X_{1},L_{2}']&=qX_{3}-\theta L_{1}X_{2}\\
[L_{2},Y_{1}]&=qY_{3}+\theta Y_{2}L_{1}\\
[L_{2}',Y_{1}]&=(q-2)Y_{3}+\theta Y_{2}L_{1}
\end{align}
\end{subequations}

\begin{subequations}
\begin{align}
&2L_{3}-\theta L_{1}L_{2}^{\prime}=\theta L_{2}L_{1}-[L_{2},L_{1}]\\
&2L_{3}'-\theta L_{2}^{\prime}L_{1}=\theta L_{1}L_{2}+[L_{2},L_{1}]\\
&\tilde{L}_3=\frac{1}{2}\theta L_1 \tilde{L}_2+ \frac{1}{2}\theta \tilde{L}_2 L_1
\end{align}
\end{subequations}

\begin{subequations}
\begin{align}
[X_{2},L_{2}]&=qX_{4}+\theta L_{1}X_{3}\\
[X_{2},L_{2}']&=(q-2)X_{4}+\theta L_{1}X_{3}\\
[L_{2},Y_{2}]&=(q+2)Y_{4}-\theta Y_{3}L_{1}\\
[L_{2}',Y_{2}]&=qY_{4}-\theta Y_{3}L_{1}
\end{align}
\end{subequations}

\begin{subequations}
\begin{align}
[L_2,L_{1}^2]&=0\  ([M_3,M_1]=0)\\
[L_3,L_{1}^2]&=0
\end{align}
\end{subequations}
\end{propo}

\begin{proof}
$[X_1, L_2']=[X_1,[X_2,Y_1]]=qX_2L_1+q^{-1}L_1X_2=qX_3-\theta L_1X_2 $. Then $[X_1, L_2]=(q+2)X_3-\theta L_1X_2$. Now we have proved (2.1a),(2.1b). Similary (2.1c), (2.1d) also hold.

On one hand,
\begin{align}
 [[X_1,L_2],Y_1]&=[(q+2)X_3-\theta L_1X_2, Y_1] \nonumber \\
 &=(q+2)L_3-\theta(L_1L_2'-q^{-1}Y_2X_2). \nonumber
\end{align}

On the other hand, 
\begin{align*}
[[X_1,L_2],Y_1]&=[X_1,[L_2,Y_1]]-[L_2,L_1] \nonumber \\
&=[X_1, qY_3+\theta Y_2L_1]-[L_2, L_1]   \nonumber \\
&=qL_3+\theta(L_2L_1+q^{-1}Y_2X_2)-[L_2,L_1]. \nonumber
\end{align*}
Compare two sides we get (2.2a) and equivalently (2.2b). (2.2c) is the sum of (2.2a) and (2.2b).

Use $[X_1, \quad]$ act on (2.2b) we can get 
\begin{align}
(q+q^{-1}+1)[X_2, L_2]=q(q+q^{-1}+1)X_4+\theta (q+q^{-1}+1)L_1X_3. \nonumber
\end{align}

Since q is not a root of unity, we get (2.3a) and then immediately (2.3b). Similarly we can get (2.3c) and (2.3d).

Use $[\quad, Y_2]$ act on (2.1a), by similar calculation we can get 

$$(q+2)[L_2, L_1^2]=\theta L_1L_3'-\theta L_3L_1.$$ 

Use $[\quad, Y_1]$ act on (2.3a) we can get 

$$q[L_2, L_1^2]=\theta L_3'L_1-\theta L_1L_3-[L_2, L_2'],$$

 i.e. $$(q-2)[L_2, L_1^2]=\theta L_3'L_1-\theta L_1L_3.$$ 

Take the sum of these two equations we can get

$$2q[L_2, L_1^2]=\theta[L_2, L_1^2].$$ 

Since q is not a root of unity, we have $[L_2, L_1^2]=0$ or equivalently $[M_3, M_1]=0$. Then (2.4b) holds by (2.2a).

\end{proof}

\begin{remar} \label{remar1}
We will use the technique in this proof frequently later, to produce new relations in imaginary root space from known relations.

\end{remar}

\begin{remar}
Comparing to the case of $U_v(\hat{sl_2})$, (2.2a) tells us $L_3$ is no longer a root vector in terms of Lyndon word theory. We will see for $n> 1$, $L_{2n+1}$ are also no longer root vectors later.

\end{remar}

\subsection{Central elements in  $K_{\ge1}/K_{>1}$}

We will first prove $M_{2m+1}^2$, $\forall m\ge 0$ commute with $L_n$, $\forall n> 0$ in $B(\mathbb{V})$.

\begin{lemma} Suppose there is a number sequence $\left\{ a_n \right\}_{n\ge 1}$ in $\mathbb{C}$ and $a_{n+2}=Aa_{n+1}+Ba_n$, for $A,B\in \mathbb{C}$. Suppose the roots $x_1$, $x_2$ of $x^2-Ax-B=0$ are different, then 

$$a_n=\frac{x_2^{n-1}(a_2-x_1a_1)-x_1^{n-1}(a_2-x_2a_1)}{x_2-x_1}.$$

\label{lemmaseries}

\end{lemma}

The following proposition plays an important role in this paper. There are several very useful corollaries after it. Moveover, in next subsection, we will write  this proposition and its corollaries in the form of series, which will be used to determine the root multiplicities completely.

\begin{propo} \label{propo:central}For $k\ge 0$,we have the following relations in $B(\mathbb{V})$. 

(a)$\tilde{L}_{4k+1}=\frac{1}{2} \theta L_1\tilde{L}_{4k}+ \frac{1}{2}\theta \tilde{L}_{4k} L_1+\frac{1}{4}[\tilde{L}_{4k-2},M_{3}]$.

(b)$[\tilde{L}_{4k},L_1]=\frac{1}{2} \theta M_3\tilde{L}_{4k-2}+\frac{1}{2} \theta \tilde{L}_{4k-2}M_3+\frac{1}{4}[\tilde{L}_{4k-4},M_5]$.

(c)$[\tilde{L}_{4k+2},L_1]=\frac{1}{2} \theta M_3\tilde{L}_{4k}+ \frac{1}{2} \theta \tilde{L}_{4k}M_3+\frac{1}{4}[\tilde{L}_{4k-2},M_5]$.

(d)$\tilde{L}_{4k+3}=\frac{1}{2} \theta L_1\tilde{L}_{4k+2}+ \frac{1}{2} \theta\tilde{L}_{4k+2} L_1+\frac{1}{4}[\tilde{L}_{4k},M_3]$.

(e)$[\tilde{L}_{n},L_1^2]=0=[L_{n},L_1^2]$ for $n=4k,4k+1,4k+2,4k+3$.

(f)$[\tilde{L}_2,\tilde{L}_{4k}]=0$ and $[\tilde{L}_2,\tilde{L}_{4k+2}]=0$.

\end{propo}

\begin{proof} We have known it holds for $k=0$ from Prop.~\ref{propo1}. Suppose it holds for k, we consider the case of k+1. We denote  anti bracket by $[\quad ,\quad ]'$, i.e. $[x,y]'=m\circ (id+c)(x\otimes y)$.

We have known $[L_{i}, L_1^2]=0$ for $i \le 4k+3$ from induction hypothesis. Then by  Lemma~\ref{lemma1}, $[X_3, Y_{4k+3}]=L_{4k+5}$ and $[X_2, Y_{4k+3}]=L_{4k+4}'$. Use $[\quad, Y_{4k+3}]$ act on (2.1a), on one hand we have 
\begin{align}
[[X_1, L_2], Y_{4k+3}]&=[(q+2)X_3-\theta L_1X_2, Y_{4k+3}]\nonumber \\
                      &=(q+2)L_{4k+5}-\theta (L_1 L_{4k+4}'-q^{-1}Y_{4k+4}X_2). \nonumber
\end{align}

On the other hand, from Lemma ~\ref{lemma1}, (2.1c) and induction hypothesis we have $[L_2, Y_{4k+3}]=qY_{4k+5}+\theta Y_{4k+4}L_1$, then
\begin{align}
 [[X_1, L_2], Y_{4k+3}]&=[X_1, qY_{4k+5}+\theta Y_{4k+4}L_1]-[L_2, L_{4k+3}]  \nonumber \\
&=qL_{4k+5}+\theta(L_{4k+4}L_1+q^{-1}Y_{4k+4}X_2)-[L_2, L_{4k+3}].\nonumber
\end{align}

Compare two equations we get $2L_{4k+5}=\theta L_1L_{4k+4}'+\theta L_{4k+4}L_1-[L_2, L_{4k+3}]$. 

Smilarly, use $[\quad, Y_{4k+2}]$ act on (2.3a), on one hand we have 
\begin{align}
[[X_2, L_2], Y_{4k+2}]&=[qX_4+\theta L_1X_3, Y_{4k+2}]  \nonumber \\
&=qL_{4k+5}'+\theta (L_1L_{4k+4}+q^{-1}Y_{4k+3}X_3).\nonumber
\end{align}
 On the other hand we have 
\begin{align}
[[X_2, L_2], Y_{4k+2}]&=[X_2, (q+2)Y_{4k+4}-\theta Y_{4k+3}L_1]-[L_2, L_{4k+3}'] \nonumber \\
&=(q+2)L_{4k+5}'-\theta (L_{4k+4}'L_1-q^{-1}Y_{4k+3}X_3)-[L_2, L_{4k+3}'].\nonumber
\end{align}

Compare both sides we get $2L_{4k+5}'=\theta L_{4k+4}'L_1+\theta L_1 L_{4k+4}+[L_2, L_{4k+3}']$.

Take the sum of the last equations of these two paragraphs we get 
\begin{align}
\tilde{L}_{4k+5}&=\frac{1}{2}\theta L_1 \tilde{L}_{4k+4}+\frac{1}{2}\theta L_1 \tilde{L}_{4k+4}+[L_2, [L_{4k+2},L_1]] \nonumber\\
&=\frac{1}{2}\theta L_1 \tilde{L}_{4k+4}+\frac{1}{2}\theta L_1 \tilde{L}_{4k+4}+[L_2, [\tilde{L}_{4k+2},L_1]] \nonumber\\
&=\frac{1}{2}\theta L_1 \tilde{L}_{4k+4}+\frac{1}{2}\theta L_1 \tilde{L}_{4k+4}+[\tilde{L}_{4k+2},M_3]. \nonumber
\end{align}
 Then (a) holds.

Take the difference of the last equations of these two paragraphs we get 
\begin{align}
[\tilde{L}_{4k+4}, L_1]&=[L_2, \tilde{L}_{4k+3}]\nonumber \\
&=[L_2, \frac{1}{2} \theta L_1\tilde{L}_{4k+2}+ \frac{1}{2} \theta\tilde{L}_{4k+2} L_1+\frac{1}{4}[\tilde{L}_{4k},M_3]] \nonumber\\
&=\frac{1}{2} \theta M_3 \tilde{L}_{4k+2}+\frac{1}{2} \theta  \tilde{L}_{4k+2}M_3+\frac{1}{4}[\tilde{L}_{4k},M_5]. \nonumber
\end{align}

Then (b) holds.

Now we have 
\begin{align}
[\tilde{L}_{4k+4},L_1^2]=&[[\tilde{L}_{4k+4},L_1],L_1] \nonumber \\
=&\frac{1}{2} \theta[M_3,[\tilde{L}_{4k+2},L_1]]'+\frac{1}{4}[[\tilde{L}_{4k},M_5],L_1]\nonumber\\
=&(\frac{1}{2} \theta)^2[M_3^{2},\tilde{L}_{4k}]+\frac{1}{2} \theta \frac{1}{4}[M_3,[\tilde{L}_{4k-2},M_5]]'\nonumber\\
&-\frac{1}{2}[\tilde{L}_{4k},M_3^2]-\frac{1}{4}[M_5, \frac{1}{2} \theta M_3\tilde{L}_{4k-2}+\frac{1}{2} \theta \tilde{L}_{4k-2}M_3+\frac{1}{4}[\tilde{L}_{4k-4},M_5] ]\nonumber\\
=&(\frac{1}{2} \theta)^2[M_3^{2},\tilde{L}_{4k}]-\frac{1}{2}[\tilde{L}_{4k},M_3^{2}]-(\frac{1}{4})^2[\tilde{L}_{4k-4},M_5^2]\nonumber\\
=&(-\frac{1}{4}\theta^2-\frac{1}{2})[\tilde{L}_{4k},M_3^{2}]-(\frac{1}{4})^2[\tilde{L}_{4k-4},M_5^2]. \nonumber
\end{align}
(for $k=0$, we have $[\tilde{L}_4,L_1^2]=\frac{1}{2}\theta[M_3, M_3]'=0.$)

Define A=$-\frac{1}{4}\theta^2-\frac{1}{2}$, B=$-(\frac{1}{4})^2$, we have 

$[\tilde{L}_{4k+4},L_1^2]=A[\tilde{L}_{4k},M_3^2]+B[\tilde{L}_{4k-4},M_5^2].$

More over similarly we have $[\tilde{L}_{4k},M_3^2]=A[\tilde{L}_{4k-4},M_5^2]+B[\tilde{L}_{4k-8},M_7^2]\cdots $ and the last one is $[\tilde{L}_8,M_{2k-1}^2]=A[\tilde{L}_4,M_{2k+1}^2]$.

\begin{align}
[\tilde{L}_4,M_{2k+1}^2]&=[[\tilde{L}_4, M_{2k+1}], M_{2k+1}]  \nonumber\\
&=[\frac{1}{2} \theta M_{2k+3}\tilde{L}_{2}+\frac{1}{2} \theta \tilde{L}_{2}M_{2k+3}, M_{2k+1}]=0.\nonumber
\end{align}

So now we have proved $[\tilde{L}_{4k+4},L_1^2]=0=[L_{4k+4},L_1^2]$ and moreover  $[\tilde{L}_{4k+4-4i},M_{2i+1}^2]=0$ for $i\ge 0$. Then automatically $[\tilde{L}_{4k+5},L_1^2]=0=[L_{4k+5},L_1^2]$.

Using induction hypothesis and the same technique in the proof of (a) and (b), we can get $[L_2,L_{4k+4}]=\theta L_1L_{4k+5}^{\prime}-\theta L_{4k+5}L_1$ and $[L_2,L_{4k+4}^{\prime}]=\theta L_{4k+5}^{\prime}L_1-\theta L_1L_{4k+5}$. Since we have known $[L_{4k+4},L_1^2]=0=[\tilde{L}_{4k+4},L_1^2]$, then $[L_2,\tilde{L}_{4k+4}]=0=[\tilde{L}_2,\tilde{L}_{4k+4}]$.

Also using induction hypothesis and the same technique in the proof of (a) and (b), we have (c),(d) holds for k+1.

The last and most difficult part is  to prove $[\tilde{L}_{4k+6},L_1^2]=0$.

We have known 
\begin{align}
[\tilde{L}_{4k+6},L_1]&=[L_2,\tilde{L}_{4k+5}]  \nonumber\\
&=\frac{1}{2} \theta M_3\tilde{L}_{4k+4}+ \frac{1}{2} \theta \tilde{L}_{4k+4}M_3+\frac{1}{4}[\tilde{L}_{4k+2},M_5],  \nonumber
\end{align}
then we have two equations about $[\tilde{L}_{4k+6},L_1^2]$:
\begin{align}
1.\ [\tilde{L}_{4k+6},L_1^2]&=[[\tilde{L}_{4k+6},L_1],L_1] \nonumber\\
=&\frac{1}{2} \theta [M_3,[\tilde{L}_{4k+4},L_1]]'+\frac{1}{4}[[\tilde{L}_{4k+2},M_5],L_1] \nonumber\\
=&(\frac{1}{2} \theta)^2[M_3^2,\tilde{L}_{4k+2}]+\frac{1}{2} \theta \frac{1}{4}[M_3,[\tilde{L}_{4k},M_5]]' \nonumber\\
&+\frac{1}{4}[\tilde{L}_{4k+2},-2M_3^2]-\frac{1}{4}[M_5,[\tilde{L}_{4k+2},L_1]]\nonumber\\
=&-\frac{1}{4}(\theta^2+2)[\tilde{L}_{4k+2},M_3^2]+\frac{1}{2} \theta \frac{1}{4}[M_3,[\tilde{L}_{4k},M_5]]'\nonumber\\
&-\frac{1}{4}\frac{1}{2}\theta (-1)[M_3,[M_5,\tilde{L}_{4k}]]'-(\frac{1}{4})^2[\tilde{L}_{4k-2},M_5^2].\nonumber
\end{align}

i.e. $[\tilde{L}_{4k+6},L_1^2]=A[\tilde{L}_{4k+2},M_3^2]+B[\tilde{L}_{4k-2},M_5^2]$.

Similarly we have $[\tilde{L}_{4k+2},M_3^2]=A[\tilde{L}_{4k-2},M_5^2]+B[\tilde{L}_{4k-6},M_7^2]\cdots$ and the last one is $[\tilde{L}_{6},M_{2k+1}^2]=-\frac{1}{4}(\theta^2+3)[\tilde{L}_{2},M_{2k+3}^2]$.

So we have $[\tilde{L}_{4k+6},L_1^2]=a_{k+2}[\tilde{L}_{2},M_{2k+3}^2]$, where the squence of number  $ a_n,n\ge1$ satisfies $a_1=1, a_2=-\frac{1}{4}(\theta^2+3), a_{n+2}=Aa_{n+1}+Ba_{n}$.

Similar to the proof of $[L_2, L_1^2]=0$, we have

$(q+2)[L_{4k+6},L_1^2]=\theta L_1L_{4k+7}'-\theta L_{4k+7}L_1-[L_2, L_{4k+6}]$,

$q[L_{4k+6},L_1^2]=\theta L_{4k+7}'L_1-\theta L_1L_{4k+7}-[L_2, L_{4k+6}']$

then we have $[\tilde{L}_{4k+6},L_1^2]=\theta L_1\tilde{L}_{4k+7}-\theta \tilde{L}_{4k+7}L_1-\frac{1}{2}[L_2,[\tilde{L}_{4k+5},L_{1}]]$.

We have known $\tilde{L}_{4k+7}=\frac{1}{2} \theta L_1\tilde{L}_{4k+6}+ \frac{1}{2} \theta \tilde{L}_{4k+6} L_1+\frac{1}{4}[\tilde{L}_{4k+4},M_3]$.

So we have
\begin{align}
2.\ [\tilde{L}_{4k+6},L_1^2]=&\theta \frac{1}{2}\theta [L_1^2,\tilde{L}_{4k+6}]+\theta \frac{1}{4}[L_1,[\tilde{L}_{4k+4},M_3]]'\nonumber\\
&-\frac{1}{2}([\tilde{L}_{4k+6},L_1^2]+[\tilde{L}_{4k+5},M_3])\nonumber\\
=&-\frac{1}{2}(\theta^2+1)[\tilde{L}_{4k+6},L_1^2]+\theta \frac{1}{4}[L_1,[\tilde{L}_{4k+4},M_3]]'\nonumber\\
&-\frac{1}{2}[\frac{1}{2}\theta L_1\tilde{L}_{4k+4}+\frac{1}{2}\theta \tilde{L}_{4k+4}L_1+\frac{1}{4}[\tilde{L}_{4k+2},M_3],M_3] \nonumber
\end{align}

i.e. $(3+\theta^2)[\tilde{L}_{4k+6},L_1^2]=-\frac{1}{4}[\tilde{L}_{4k+2},M_3^2]$

i.e. $(3+\theta^2)[\tilde{L}_{4k+6},L_1^2]=-\frac{1}{4}a_{k+1}[\tilde{L}_{2},M_{2k+3}^2]$.

In summary we have

$[\tilde{L}_{4k+6},L_1^2]=a_{k+2}[\tilde{L}_{2},M_{2k+3}^2],$

$(3+\theta^2)[\tilde{L}_{4k+6},L_1^2]=-\frac{1}{4}a_{k+1}[\tilde{L}_{2},M_{2k+3}^2]$

so if 
$
\left | \begin{matrix}

1& a_{k+2} \\
3+\theta^2 & -\frac{1}{4}a_{k+1}\\

\end{matrix} \right | \neq 0
$, then we have $[\tilde{L}_{4k+6},L_1^2]=0$ and then $[L_{4k+6},L_1^2]=0$. With Lemma ~\ref{lemmaseries} we can easily see this holds when $q$ is not a root of unity, hence $[\tilde{L}_{4k+6},L_1^2]=0$ and moreover $[\tilde{L}_{4k+6-4i},M_{2i+1}^2]=0$ for $i\ge 0$.

Then automatically $[\tilde{L}_{4k+7},L_1^2]=0=[L_{4k+7},L_1^2]$ and similar to the proof of $[L_2,\tilde{L}_{4k+4}]=0=[\tilde{L}_2,\tilde{L}_{4k+4}]$   before, we can get $[L_2,\tilde{L}_{4k+6}]=0=[\tilde{L}_2,\tilde{L}_{4k+6}]$.

\end{proof}
\begin{corol} \label{mrelations}
$[M_{2m+1},M_{2n+1}]=0 $ if $m+n$ is odd, $[M_{2m+1},M_{2n+1}]=(-1)^{\frac{m-n}{2}}2M_{m+n+1}^2 $ if $m+n$ is even and $M_{2m+1}^2$ commute with all $L_{n}$.

\end{corol}
\begin{proof}
By an easy induction on $m+n$ we can get the formula of $[M_{2m+1},M_{2n+1}]$. Then we can see $[L_1, M_{2n+1}^2]=0$ for all $n$. By the proof of Prop.~\ref{propo:central}, we know $M_{2n+1}^2$ commute with all $\tilde{L}_{2m}$. Then by Prop.~\ref{propo:central} $M_{2n+1}^2$ commute with all $\tilde{L}_{n}$ and this is equivalent to $M_{2n+1}^2$ commute with all $L_{n}$.
\end{proof}

\begin{corol} The comultiplication of $L_{n}$ and $L_{n}'$ in  $K_{\ge 1}/K_{>1}$ are as the following. \label{corol:comulformula}
\begin{align}
\Delta (L_{2n})=&L_{2n}\otimes 1 -\theta L_{2n-1}'\otimes L_1+\theta L_{2n-2}\otimes L_2 -\theta L_{2n-3}'\otimes L_3+... \nonumber\\
&-\theta L_1'\otimes L_{2n-1}+1\otimes L_{2n} \nonumber\\
\Delta (L_{2n}')=&L_{2n}'\otimes 1 -\theta L_{2n-1}\otimes L_1'+\theta L_{2n-2}'\otimes L_2' -\theta L_{2n-3}\otimes L_3'+...\nonumber\\
&-\theta L_1\otimes L_{2n-1}'+1\otimes L_{2n}'. \nonumber\\
\Delta (L_{2n+1})=&L_{2n+1}\otimes 1+\theta L_{2n}'\otimes L_1+\theta L_{2n-1}\otimes L_2+\theta L_{2n-2}'\otimes L_3+\cdots\nonumber\\
&+\theta L_1\otimes L_{2n}+1\otimes L_{2n+1}\nonumber\\
\Delta (L_{2n+1}')=&L_{2n+1}'\otimes 1+\theta L_{2n}\otimes L_1'+\theta L_{2n-1}'\otimes L_2'+\theta L_{2n-2}\otimes L_3'+\cdots\nonumber\\
&+\theta L_1'\otimes L_{2n}'+1\otimes L_{2n+1}'\nonumber
\end{align}
\end{corol}

\begin{proof}
By $L_1^2$ commutes with $L_n$ and an easy induction, we can get the expressions of $\Delta (L_{2n})$ and $\Delta (L_{2n+1})$. Then we can get the expressions of $\Delta (L_{2n}')$ and $\Delta (L_{2n+1}')$ by Lemma ~\ref{lemma1}.
\end{proof}

\begin{remar} The comultiplication formula in the subquotient $K_{\ge 1}/K_{>1}$ of $L_n$ in $U_{v}(\hat{sl_2})^{+}$ is $\Delta(L_n)=\sum_{i=0}^{n}\theta L_{n-i}\otimes L_{i}$, corresponding to a result of  Prop. 4.3 in \cite{quantumaffinealgebras}.

\end{remar}

\begin{corol}\label{corol:Lnonzero}
$L_{n}$ and $L_{n}'$ are not 0  and  linearly independent in $B(\mathbb{V})$, for $n\ge 0$.
\end{corol}

\begin{proof}
Follow Corollary ~\ref{corol:comulformula} and an easy induction.
\end{proof}

\begin{lemma}We have the following results using differential operators.\label{lemma:diff} 

(a) $\partial_2^R(L_1^2)=-q^{-1}\theta X_2$, $\partial_2^R(Y_{2n+1})=\theta L_{2n}'$, $\partial_2^R(Y_{2n})=-\theta L_{2n-1}'$, $\partial_2^R(L_{2n})=-\theta (X_1 L_{2n-1}'-q^{-2}L_{2n-1}'X_1)$, $\partial_2^R(L_{2n+1})=\theta (X_1 L_{2n}'-q^{-2}L_{2n}'X_1)$.

(b) $\partial_1^L(L_1^2)=q^{-1}\theta Y_2$, $\partial_1^L(X_{2n+1})=\theta L_{2n}$, $\partial_1^L(X_{2n})=\theta L_{2n-1}'$, $\partial_1^L(L_{2n}')=\theta (L_{2n-1}'Y_1+q^{-2}Y_1 L_{2n-1}')$, $\partial_1^L(L_{2n+1})=\theta (L_{2n}Y_1-q^{-2}Y_1L_{2n})$.

\end{lemma}

\begin{proof}
By Lemma ~\ref{lemma1}, Prop.~\ref{propo:central} and induction.
\end{proof}

\begin{corol}\label{corol:realrootvectors}
In terms of Lyndon word, $X_n$ and $Y_n$ are the root vectors in corresponding degrees.
\end{corol}

\begin{proof}
By  Lemma ~\ref{lemma:realrootvectors}, Lemma ~\ref{lemma:diff}, Corollary ~\ref{corol:Lnonzero}.
\end{proof}

\begin{corol}
As an algebra $K_{\ge1}/K_{>1}$ is isomorphic to the subalgebra of $B(\mathbb{V})$ generated by $L_n$ for $n> 0$. Then $M_{2m+1}^2$ are central elements in  $K_{\ge1}/K_{>1}$.

\end{corol}

\begin{proof}
By Corollary~\ref{corol:realrootvectors} we have $X_n$ and $Y_n$ are real root vectors, then the only possible imaginary root vectors are $L_n$ for $n\ge 1$ and their iterated commutators by Lyndon word theory. By the expression of $\Delta(Y_n)$ in Lemma ~\ref{subquotientlemma}, we get $K_{\ge 1}$ is the subalgebra generated by $X_n$ and $L_n$ and $K_{\ge 1}=K_{>1}\oplus K_{=1}$ where $K_{=1}$ is the subalgebra generated by $L_{n}$ and $K_{>1}$ is the subalgebra $K_{\ge 1}\cap B_{>1}$. This means as an algebra $K_{\ge1}/K_{>1}$ is isomorphic to the subalgebra in $B(\mathbb{V})$ generated by $L_n$.

\end{proof}

\begin{remar}
In summary, $K_{\ge1}/K_{>1}$ is a braided Hopf algebra whose algebra structure is nothing new but having a manageable coalgebra structure.
\end{remar}

\subsection{Relations in the form of series}In this subsection we will describe the relations in $B(\mathbb{V})$ in terms of series, which will be powerful  in proofs of main theorems in this paper.

Let $\theta':=q+q^{-1}$, $\alpha=\sqrt{\frac{iq}{2}}$, $\beta=\sqrt{\frac{iq^{-1}}{2}}$.

Let $a_1=1$, $a_2=\frac{1}{2}\theta$, $a_{n+2}=\frac{1}{2}\theta a_{n+1}+\frac{1}{4}a_n$, then by Lemma ~\ref{lemmaseries} we have $a_n=\frac{(\frac{1}{2}q)^n-(-\frac{1}{2}q^{-1})^n}{\frac{1}{2}(q+q^{-1})}$.

Let 
\begin{align}
\tilde{L}(u)&:=\sum_{n=0}^{\infty}\theta \tilde{L}_{2n}u^{2n}, X(u):=\sum_{n=0}^{\infty}2\theta M_{2n+1}^2 u^{4n+2},\nonumber\\
A(u)&:=\sum_{n=0}^{\infty}\theta a_{n+1}M_{2n+1}u^{2n+1},B(u):=\sum_{n=1}^{\infty}\theta a_{n}M_{2n+1}u^{2n+1}, \nonumber\\
A'(u)&:=\frac{1}{i}A(iu), B'(u):=iB(iu).\nonumber
\end{align}
Sometimes for convenience we just write $A$, $B$, $A'$, $B'$, $X$.

\begin{propo} \label{seriesrelations}
We have the following relations in the form of series.

(a) $A(u)\tilde{L}(u)=\tilde{L}(u)A'(u)=\sum_{n=0}^{\infty}\theta \tilde{L}_{2n+1}u^{2n+1}.$

(b) $B(u)\tilde{L}(u)=\tilde{L}(u) B'(u)=\sum_{n=1}^{\infty}\theta \hat{L}_{2n+1}u^{2n+1}.$

(c) $[A,B]=0=[A',B'].$

(d) $A^2=\frac{q}{\theta'i}X(\alpha u)-\frac{q^{-1}}{\theta'i}X(\beta u)$, $\frac{1}{4}B^2=\frac{q}{\theta'i}X(\beta u)-\frac{q^{-1}}{\theta'i}X(\alpha u)$.

(e) $\frac{1}{2}[A(u), L_1 u]=\frac{1}{\theta}(\frac{1}{4}B^2+A^2).$
\end{propo}

\begin{proof}
(a), (b) follow from Prop. ~\ref{propo:central} and  (c), (d), (e) follow from Corollary ~\ref{mrelations}.

\end{proof}

\subsection{Relations between $X_n$ and $L_n$ and $Y_n$ and $L_n$.}

In this section we will use differential operators to get the expressions of  

$[X_1,L_{2n}],\ [X_1,L_{2n}^{'}],\  [X_2,L_{2n}],\ [X_2,L_{2n}^{'}]$ and

 $[L_{2n},Y_1],\ [L_{2n}^{'},Y_1],\ [L_{2n},Y_2],\ [L_{2n}^{'},Y_2]$.

In general, by Lemma ~\ref{lemma1} and  $L_1^2$ lies in the center of the subalgebra generated by all $L_n$, one can get $[X_{k},L_{2n}]$ for $k\ge 3$ and others immediately.

Define two sequences of polynomials $P_n$ and $Q_n$:

$P_0=1, P_n=q^n+q^{n-1}+(-1)^{n-1}P_{n-1}$,

$Q_0=1, Q_n=q^n-q^{n-1}+(-1)^nQ_{n-1}$. 
We have

$P_n=2q^{2 \lfloor \frac{n}{2}\rfloor}\frac{1-(-q^{-2})^{\lfloor \frac{n}{2}\rfloor+1}}{1+q{-2}}+(-1)^{n+1}q^n,$

$Q_n=(-1)^n 2q^{2 \lfloor \frac{n}{2}\rfloor}\frac{1-(-q^{-2})^{\lfloor \frac{n}{2}\rfloor+1}}{1+q{-2}}+(-1)^{n+1}q^n$ and  $P_{2n}=Q_{2n}$.\\

The next proposition is for the proof of $[\tilde{L}_{2n}, \tilde{L}_{2m}]=0$, $\forall n,m\ge1$ in Prop. ~\ref{propo:someimaginaryrootvectorscomm}.
\begin{propo} \label{propo:realandim} We have the following relations in $B(\mathbb{V})$.

$[X_1, L_{2n}]=\sum_{i=0}^{n-1}P_{2n-1-2i}\theta L_{2i}X_{2n+1-2i}-\sum_{i=0}^{n-1}P_{2n-2-2i}\theta L_{2i+1}'X_{2n-2i}$,

$[X_1, L_{2n}']=\sum_{i=0}^{n-1}q^{2n-1-2i}\theta L_{2i}'X_{2n+1-2i}-\sum_{i=0}^{n-1}q^{2n-2-2i}\theta L_{2i+1}X_{2n-2i}$,

$[X_2, L_{2n}]=\sum_{i=0}^{n-1}q^{2n-1-2i}\theta L_{2i}X_{2n+2-2i}+\sum_{i=0}^{n-1}q^{2n-2-2i}\theta L_{2i+1}'X_{2n-2i+1}$,

$[X_2, L_{2n}']=\sum_{i=0}^{n-1}Q_{2n-1-2i}\theta L_{2i}'X_{2n+2-2i}+\sum_{i=0}^{n-1}Q_{2n-2-2i}\theta L_{2i+1}X_{2n-2i+1}$,
\\

$[L_{2n}, Y_1]=\sum_{i=0}^{n-1}q^{2n-1-2i}\theta Y_{2n+1-2i}L_{2i}+\sum_{i=0}^{n-1}q^{2n-2-2i}\theta Y_{2n-2i}L_{2i+1}$,

$[L_{2n}', Y_1]=\sum_{i=0}^{n-1}Q_{2n-1-2i}\theta Y_{2n+1-2i}L_{2i}'+\sum_{i=0}^{n-1}Q_{2n-2-2i}\theta Y_{2n-2i}L_{2i+1}'$,

$[L_{2n}, Y_2]=\sum_{i=0}^{n-1}P_{2n-1-2i}\theta Y_{2n+2-2i}L_{2i}-\sum_{i=0}^{n-1}P_{2n-2-2i}\theta Y_{2n-2i+1}L_{2i+1}$,

$[L_{2n}', Y_2]=\sum_{i=0}^{n-1}q^{2n-1-2i}\theta Y_{2n+2-2i}L_{2i}'-\sum_{i=0}^{n-1}q^{2n-2-2i}\theta Y_{2n-2i+1}L_{2i+1}'$.

\end{propo}
\begin{proof}We use induction. By Prop. ~\ref{propo1} we know the proposition holds for $n=1$. Suppose the proposition holds for $n$, we consider the case of $n+1$. 

We have $[L_{2n+1}, L_1^2]=0$. Hence in $B(\mathbb{V})$, by Lemma ~\ref{lemma:diff} we have 
\begin{align}
0&=\partial_2^R([L_{2n+1}, L_1^2]) \nonumber \\
&=L_{2n+1}(-q^{-1}\theta X_2)+\partial_2^R(L_{2n+1})L_1^2-(L_1^2\partial_2^R(L_{2n+1})+q^{-1}\theta X_2 L_{2n+1})\nonumber \\
&=-q^{-1}\theta [X_2, L_{2n+1}]+[\partial_2^R(L_{2n+1}), L_1^2],  \nonumber 
\end{align}
i.e. $[X_2, L_{2n+1}]=q[[X_1, L_{2n}']+q^{-1}\theta L_{2n}'X_1, L_1^2].$

Using induction hypothesis, we have 

$[X_2, L_{2n+1}]=\sum_{i=0}^{n}q^{2n-2i}\theta L_{2i}'X_{2n+3-2i}-\sum_{i=0}^{n-1}q^{2n-1-2i}\theta L_{2i+1}X_{2n-2i+2}.$

Then\begin{align}
[X_1, L_{2n+2}']&=[X_1, [X_2, Y_{2n+1}]]=q[X_2, L_{2n+1}]-\theta L_{2n+1}X_2 \nonumber\\
&=\sum_{i=0}^{n}q^{2n-2i+1}\theta L_{2i}'X_{2n+3-2i}-\sum_{i=0}^{n-1}q^{2n-2i}\theta L_{2i+1}X_{2n-2i+2}-\theta L_{2n+1}X_2\nonumber\\
&=\sum_{i=0}^{n}q^{2n-2i+1}\theta L_{2i}'X_{2n+3-2i}-\sum_{i=0}^{n}q^{2n-2i}\theta L_{2i+1}X_{2n-2i+2}.\nonumber
\end{align}

$[X_1, L_{2n+2}]=[X_1, L_{2n+2}']+[X_1, [L_{2n+1}, L_1]]$, so we need to calculate $[X_1, L_{2n+1}]$. 
\begin{align}
[X_1, L_{2n+1}]=&[X_1, L_{2n+1}']-[X_1, [L_{2n}, L_1]]    \nonumber \\
=&[X_1, [X_2, Y_{2n}]]-[[X_1, L_{2n}], L_1]+[X_2, L_{2n}]  \nonumber \\
=&(q+1)[X_2, L_{2n}]+\theta L_{2n}X_2-[[X_1, L_{2n}], L_1]\nonumber \\
=&\sum_{i=0}^{n-1}(q^{2n-2i}+q^{2n-1-2i})\theta L_{2i}X_{2n+2-2i}\nonumber \\
&+\sum_{i=0}^{n-1}(q^{2n-1-2i}+q^{2n-2-2i})\theta L_{2i+1}'X_{2n-2i+1}\nonumber \\
&+\theta L_{2n}X_2-\left(\sum_{i=0}^{n-1}P_{2n-1-2i}\theta (L_{2i}X_{2n+2-2i}+\hat{L}_{2i+1}X_{2n+1-2i})\nonumber \right.\\
&\left.-\sum_{i=0}^{n-1}P_{2n-2-2i}\theta (L_{2i+1}'X_{2n-2i+1}-\hat{L}_{2i+2}X_{2n-2i})\right), \nonumber 
\end{align}
by the definition of $P_n$, we have 
\begin{align}
[X_1, L_{2n+2}]=&\sum_{i=0}^{n-1}P_{2n-2i}\theta L_{2i}X_{2n+2-2i}-\sum_{i=1}^{n}P_{2n-2i}\theta \hat{L}_{2i}X_{2n+2-2i}+\theta L_{2n}X_2 \nonumber \\
&+\sum_{i=0}^{n-1}P_{2n-1-2i}\theta L_{2i+1}'X_{2n+1-2i}-\sum_{i=0}^{n-1}P_{2n-1-2i}\theta \hat{L}_{2i+1}X_{2n+1-2i} \nonumber \\
=&\sum_{i=0}^{n}P_{2n-2i}\theta L_{2i}'X_{2n+2-2i}+\sum_{i=0}^{n-1}P_{2n-1-2i}\theta L_{2i+1}X_{2n+1-2i}.\nonumber 
\end{align}
Then 
\begin{align}
[X_1, L_{2n+2}]&=[X_1, L_{2n+2}']+[X_1, [L_{2n+1}, L_1]] \nonumber \\
&=[X_1, L_{2n+2}']+[[X_1, L_{2n+1}], L_1]]+[X_2, L_{2n+1}] \nonumber \\
&=(q+1)[X_2, L_{2n+1}]-\theta L_{2n+1}X_2+[[X_1, L_{2n+1}], L_1]]. \nonumber 
\end{align}

All the terms are computable now and after similar compuitation to last paragraph we have 

$
[X_1, L_{2n+2}]=\sum_{i=0}^{n}P_{2n+1-2i}\theta L_{2i}X_{2n+3-2i}-\sum_{i=0}^{n}P_{2n-2i}\theta L_{2i+1}'X_{2n+2-2i}.
$

$[X_2, L_{2n+2}]$ and $[X_2, L_{2n+2}']$ and $[L_{2n+2}, Y_1]$, $[L_{2n+2}', Y_1]$, $[L_{2n+2}, Y_2]$, $[L_{2n+2}', Y_2]$ follow from a completely similar argument, we omit the detail.

\end{proof}

Equation $[\tilde{L}_{2m},\tilde{L}_{2n}]=0$ is equivalent to $[L_{2m},L_{2n}]+[L_{2m}',L_{2n}]+[L_{2m},L_{2n}']+[L_{2m}',L_{2n}']=0$. So we need to know the expressions of each summand.

\begin{lemma}\label{lemma:tildeandhat}
Suppose $[\tilde{L}_{2k-2}, \tilde{L}_{2N}]=0$ for $\forall 2\le k\le n$ and for $\forall N>0$, then $[\tilde{L}_{2k-1}, \hat{L}_{2N+1}]=[\hat{L}_{2k-1}, \tilde{L}_{2N+1}]$ for all $1\le k\le n$ and $\forall N>0$.
\end{lemma}

\begin{proof}
We only need to prove 
\begin{align}
[\tilde{L}_{2n-1}, [\tilde{L}_{2N}  ,L_1]]=[[\tilde{L}_{2n-2}  ,L_1], \tilde{L}_{2N+1}].  \nonumber
\end{align}
By Prop. ~\ref{propo:central},
\begin{align}
LHS&=\left[\frac{1}{2}\theta L_1\tilde{L}_{2n-2}+\frac{1}{2}\theta \tilde{L}_{2n-2}L_1+\frac{1}{4}[\tilde{L}_{2n-4},M_3],[\tilde{L}_{2N}  ,L_1]\right] \nonumber\\
&=\frac{1}{2}\theta \left[L_1,[\tilde{L}_{2n-2},[\tilde{L}_{2N}  ,L_1]]\right]'+\frac{1}{4}\left[[\tilde{L}_{2n-4},M_3],[\tilde{L}_{2N}  ,L_1]\right].\nonumber
\end{align}

\begin{align}
RHS&=\left[[\tilde{L}_{2n-2}  ,L_1], \frac{1}{2}\theta L_1\tilde{L}_{2N}+\frac{1}{2}\theta \tilde{L}_{2N}L_1+\frac{1}{4}[\tilde{L}_{2N-2},M_3]\right]\nonumber\\
&=-\frac{1}{2}\theta \left[L_1,[[\tilde{L}_{2n-2}  ,L_1],\tilde{L}_{2N}]\right]'+\frac{1}{4}\left[[\tilde{L}_{2n-2}  ,L_1]  , [\tilde{L}_{2N-2},M_3] \right]\nonumber
\end{align}
Since $[\tilde{L}_{2n-2}, \tilde{L}_{2N}]=0$, we only need to prove 
\begin{align}
\left[[\tilde{L}_{2n-4},M_3],[\tilde{L}_{2N}  ,L_1]\right]=\left[[\tilde{L}_{2n-2}  ,L_1]  , [\tilde{L}_{2N-2},M_3] \right].\nonumber
\end{align}
Use the expressions of $[\tilde{L}_{2N}  ,L_1]$ and $[\tilde{L}_{2n-2}  ,L_1]$ in Prop. ~\ref{propo:central} and similar calculation, we get we only need to  prove 
\begin{align}
\left[[\tilde{L}_{2n-4},M_3],[\tilde{L}_{2N-4}  ,M_5]\right]=\left[[\tilde{L}_{2n-6}  ,M_5]  , [\tilde{L}_{2N-2},M_3] \right]. \nonumber
\end{align}
Repeat these steps this lemma will be proved.
\end{proof}

\begin{lemma} Suppose $[\tilde{L}_{2k},\tilde{L}_{2N}]=0$ for $k\le n-1$ and for all $N>0$, then $[\tilde{L}_{2n},\tilde{L}_{2N}]=0$ for all $N>0$ is equivalent to 

$-[[\tilde{L}_{2n-2},\tilde{L}_{2N+1}],L_1]+[[\tilde{L}_{2N+2},\tilde{L}_{2n-3}],L_1]=0$ for all $N>0$.
\end{lemma}

\begin{proof}
Consider $[[X_1, L_{2n}], Y_{2N}]$, using the same technique as Prop.~\ref{propo:central}, we can get
\begin{align}
 [L_{2n}, L_{2N}]=&\sum_{i=0}^{n-1}P_{2n-1-2i}\theta [L_{2n+2N-2i}, L_{2i}]  \nonumber \\
&+\sum_{i=0}^{n-1}P_{2n-2-2i}\theta(L_{2i+1}'L_{2n+2N-2i-1}'-L_{2n+2N-2i-1}L_{2i+1}). \nonumber
\end{align}
Note that in this calculation we do not need to  care about terms with form $Y_aL_bX_c$ or $Y_aL_b'X_c$. All  terms of this form would be cancelled at last by Corollary ~\ref{corol:Lnonzero} and Corollary ~\ref{corol:realrootvectors}.

Similarly consider $[[X_1, L_{2n}'],Y_{2N}]$ we can get 
\begin{align}
[L_{2n}', L_{2N}]=&\sum_{i=0}^{n-1}q^{2n-1-2i}\theta [L_{2n+2N-2i}, L_{2i}'] \nonumber \\
&+\sum_{i=0}^{n-1}q^{2n-2-2i}\theta (L_{2i+1}L_{2n+2N-1-2i}'-L_{2n+2N-1-2i}L_{2i+1}').\nonumber 
\end{align}

Simlary consider $[[X_2, L_{2n}],Y_{2N-1}]$ we can get 
\begin{align}
[L_{2n}, L_{2N}']=&\sum_{i=0}^{n-1}q^{2n-1-2i}\theta [L_{2n+2N-2i}', L_{2i}]  \nonumber \\
&+\sum_{i=0}^{n-1}q^{2n-2-2i}\theta (L_{2n+2N-1-2i}'L_{2i+1}-L_{2i+1}'L_{2n+2N-1-2i}).\nonumber
\end{align}

Simlary consider $[[X_2, L_{2n}'],Y_{2N-1}]$ we can get 
\begin{align}
[L_{2n}', L_{2N}']=&\sum_{i=0}^{n-1}Q_{2n-1-2i}\theta [L_{2n+2N-2i}', L_{2i}']  \nonumber \\
&+\sum_{i=0}^{n-1}Q_{2n-2-2i}\theta (L_{2n+2N-1-2i}'L_{2i+1}'-L_{2i+1}L_{2n+2N-1-2i}). \nonumber
\end{align}

Then by the construction of $P_n,\ Q_n$ and $[\tilde{L}_{2k},\tilde{L}_{2N}]=0$ for $k\le n-1$ and for all $N>0$, we have
\begin{align}
&[L_{2n}, L_{2N}]+[L_{2n}', L_{2N}]+[L_{2n}', L_{2N}]+[L_{2n}', L_{2N}']  \nonumber \\
=&\sum_{i=0}^{n-1}(P_{2n-2-2i}+q^{2n-2-2i})\theta ([L_{2n+2N-2i}, L_{2i}]-[L_{2n+2N-2i}', L_{2i}']\nonumber \\
&+\sum_{i=0}^{n-1}P_{2n-2-2i}\theta([L_{2i+1}', L_{2n+2N-2i-1}']-[L_{2i+1}, L_{2n+2N-2i-1}])\nonumber \\
&+\sum_{i=0}^{n-1}q^{2n-2-2i}\theta([L_{2i+1}, L_{2n+2N-2i-1}']-[L_{2i+1}', L_{2n+2N-2i-1}]).\nonumber
\end{align}
We have the following identities by definition and Lemma ~\ref{lemma:tildeandhat}:
\begin{align}
&[L_{2m}, L_{2n}]-[L_{2m}', L_{2n}']=[\hat{L}_{2m}, \tilde{L}_{2n}]+[\tilde{L}_{2m}, \hat{L}_{2n}], \ \forall m,n\ge0.\nonumber \\
&[L_{2i+1}', L_{2n+2N-2i-1}']-[L_{2i+1}, L_{2n+2N-2i-1}] \nonumber\\
=&[\tilde{L}_{2i+1}, \hat{L}_{2n+2N-2i-1}]+[\hat{L}_{2i+1}, \tilde{L}_{2n+2N-2i-1}] \nonumber\\
=&2[\tilde{L}_{2i+1}, \hat{L}_{2n+2N-2i-1}],\nonumber \\
&[L_{2i+1}, L_{2n+2N-2i-1}']-[L_{2i+1}', L_{2n+2N-2i-1}]\nonumber\\
=&[\tilde{L}_{2i+1}, \hat{L}_{2n+2N-2i-1}]-[\hat{L}_{2i+1}, \tilde{L}_{2n+2N-2i-1}]=0, \ \forall 0\le i\le n-1.\nonumber 
\end{align}

Then 
\begin{align}
&[L_{2n}, L_{2N}]+[L_{2n}', L_{2N}]+[L_{2n}', L_{2N}]+[L_{2n}', L_{2N}']  \nonumber \\
=&\sum_{i=0}^{n-1}(P_{2n-2-2i}+q^{2n-2-2i})\theta ([\hat{L}_{2n+2N-2i}, \tilde{L}_{2i}]+[\tilde{L}_{2n+2N-2i}, \hat{L}_{2i}])\nonumber \\
&+2\sum_{i=0}^{n-1}P_{2n-2-2i}\theta[\tilde{L}_{2i+1}, \hat{L}_{2n+2N-2i-1}]\nonumber \\
=&\sum_{i=0}^{n-1}(P_{2n-2-2i}+q^{2n-2-2i})\theta \left(-[\tilde{L}_{2n+2N-2i-1}, \hat{L}_{2i+1}]\nonumber \right.\\
&\left.+[L_1, [\tilde{L}_{2n+2N-2i-1}, \tilde{L}_{2i}]]+[[\tilde{L}_{2n+2N-2i}, \tilde{L}_{2i-1}], L_1]+[\tilde{L}_{2i-1}, \hat{L}_{2n+2N-2i+1}]\right)\nonumber \\
&+2\sum_{i=0}^{n-1}P_{2n-2-2i}\theta[\tilde{L}_{2i+1}, \hat{L}_{2n+2N-2i-1}].\nonumber 
\end{align}
By the construction of $P_n$, we get 
$$P_{2n-2-2i}=q^{2n-2-2i}-q^{2n-4-2i}-P_{2n-4-2i}.$$ 
Then
\begin{align}
 &[L_{2n}, L_{2N}]+[L_{2n}', L_{2N}]+[L_{2n}', L_{2N}]+[L_{2n}', L_{2N}']\nonumber \\
&=\sum_{i=0}^{n-1}(P_{2n-2-2i}+q^{2n-2-2i})\theta ([L_1, [\tilde{L}_{2n+2N-2i-1}, \tilde{L}_{2i}]]+[[\tilde{L}_{2n+2N-2i}, \tilde{L}_{2i-1}], L_1]).\nonumber 
\end{align}
When $q$ is not a root of unity, it is easy to see $(P_{2n-2-2i}+q^{2n-2-2i})\theta$ is not 0. Then by induction hypothesis we can see 
$$[L_{2n}, L_{2N}]+[L_{2n}', L_{2N}]+[L_{2n}', L_{2N}]+[L_{2n}', L_{2N}']=0$$
 is equivalent to 
$$-[[\tilde{L}_{2n-2},\tilde{L}_{2N+1}],L_1]+[[\tilde{L}_{2N+2},\tilde{L}_{2n-3}],L_1]=0$$ for $\forall N\ge 0$.

\end{proof}

Denote $D_{2k}=-[[\tilde{L}_{2k},\tilde{L}_{2N+1}],L_1]+[[\tilde{L}_{2N+2},\tilde{L}_{2k-1}],L_1]$.\\

Define $\bar{L}_{2n}$ by the following formal power series:

$$\theta \sum_{n=0}^{\infty}\tilde{L}_{2n}u^{2n}=\ex({\theta \sum_{n=1}^{\infty}\bar{L}_{2n}u^{2n}}).$$

Define $C_k(P(u))$ by the coefficient of $u^{k}$ in power series $P(u)$.

\begin{propo}\label{propo:someimaginaryrootvectorscomm} There exists a series of number ${R_{2k}}$ such that $[\bar{L}_{2k},L_1]=R_{2k}M_{2k+1}$ and $[\tilde{L}_{2m},\tilde{L}_{2n}]=0$ for $k,m,n \ge 1$. $R_{2k}$ is defined by 
\begin{align}
\theta R_{2k}=&C_{2k-2}\left(\frac{1}{2}\theta \right) + C_{2k-2}\left(\frac{1}{2}\theta \ex (\sum_{n=1}^{k-1}\theta R_{2n}u^{2n})\right) \nonumber \\
&+C_{2k-4}\left(\frac{1}{4}(\ex (\sum_{n=1}^{k-1}\theta R_{2n}u^{2n})-1)\right) -C_{2k}\left(\ex (\sum_{n=1}^{k-1}\theta R_{2n}u^{2n})\right).\nonumber 
\end{align}
Moreover, we have $R_{4k}=0$ and $R_{4k+2}=\frac{1}{2^{2k}}\frac{[2k+1]_q}{2k+1}$.
\end{propo}

\begin{proof}
We will prove by induction on $n$ that $[\bar{L}_{2n},L_1]=R_{2n}M_{2n+1}$  for some ${R_{2n}}$ and $[\tilde{L}_{2n},\tilde{L}_{2N}]=0$ for $\forall N\ge 0$.

For $n=1$, we have know in Section 1 that $[\bar{L}_2,L_1]=M_3$ which means $R_2=1$ and coincide with the formula of $R_{2k}$ when $k=1$, and $[\tilde{L}_2,\tilde{L}_{2N}]=0$ for $\forall N\ge 0$. Now suppose this propsition holds for $n\le k$, i.e we have known $[\bar{L}_{2n},L_1]=R_{2n}M_{2n+1}$ and $[\tilde{L}_{2n},\tilde{L}_{2N}]=0$  for $n\le k$ and $\forall N\ge0$, we consider $n=k+1$.

First we will prove $[\bar{L}_{2k+2},L_1]=R_{2k+2}M_{2k+3}$ for some $R_{2k+2}$.

By Prop.~\ref{propo:central} we know 
\begin{align}
\theta \sum_{n=0}^{\infty}[\tilde{L}_{2n}u^{2n},L_1]=&\frac{1}{2}\theta M_3\theta \sum_{n=0}^{\infty}\tilde{L}_{2n-2}u^{2n}\nonumber\\
&+\frac{1}{2}\theta \theta \sum_{n=0}^{\infty}\tilde{L}_{2n-2}u^{2n}M_3+\frac{1}{4}[\theta \sum_{n=0}^{\infty}\tilde{L}_{2n-4}u^{2n},M_5]\nonumber
\end{align}
i.e
\begin{align}
\theta \sum_{n=0}^{\infty}&[\tilde{L}_{2n}u^{2n},L_1]=\frac{1}{2}\theta u^2 M_3\theta \sum_{n=0}^{\infty}\tilde{L}_{2n-2}u^{2n-2} \nonumber\\
&+\frac{1}{2}\theta u^2 \theta \sum_{n=0}^{\infty}\tilde{L}_{2n-2}u^{2n-2}M_3+\frac{1}{4}u^4[\theta \sum_{n=0}^{\infty}\tilde{L}_{2n-4}u^{2n-4},M_5]\nonumber
\end{align}
i.e
\begin{align}
[\ex(\theta \sum_{n=1}^{\infty}&\bar{L}_{2n}u^{2n}),L_1]=\frac{1}{2}\theta u^2 M_3\ex(\theta \sum_{n=1}^{\infty}\bar{L}_{2n}u^{2n})\nonumber\\
&+\frac{1}{2}\theta u^2 \ex(\theta \sum_{n=1}^{\infty}\bar{L}_{2n}u^{2n})M_3+\frac{1}{4}u^4[\ex(\theta \sum_{n=1}^{\infty}\bar{L}_{2n}u^{2n}),M_5].\nonumber
\end{align}
Note that $\ex(\theta \sum_{n=1}^{\infty}\bar{L}_{2n}u^{2n})=\ex(\theta \sum_{n=2}^{\infty}\bar{L}_{2n}u^{2n})\ex(\theta \bar{L}_2 u^2)$, we denote $\ex(\theta \sum_{n=2}^{\infty}\bar{L}_{2n}u^{2n})$ by $\ex(2)$ and in general $\ex(\theta \sum_{n=k}^{\infty}\bar{L}_{2n}u^{2n})$ by $\ex(k)$ for convenience.

It is easy to get $\ex(\theta \bar{L}_{2n} u^2)L_1=\ex(ad_{\theta \bar{L}_{2n}}u^{2n})\bullet L_1 \ex(\theta \bar{L}_{2n} u^2)$. Denote $ad_{\theta \bar{L}_{2n}}u^{2n}$ by $\mathbb{M}_{2n}$ for further convenience, where $ad_{\theta \bar{L}_{2n}}$ means the operator $[\theta \bar{L}_{2n},\ \ \ ]$.

Now LHS=$[\ex(2),L_1]\ex(\theta \bar{L}_2u^2)+\ex(2)\left(\ex(\mathbb{M}_2)-1\right)\bullet L_1\ex(\theta \bar{L}_2u^2)$. Keep doing similar calculation under our induction hypothesis, we will get LHS=

$\left([\ex(k+1),L_1]+\ex(k+1)\left(\ex(\sum_{n=1}^{k}\mathbb{M}_{2n})-1\right)\bullet L_1\right)\prod_{n=1}^{k}\ex(\theta \bar{L}_{2n}u^{2n})$.

Smilarly, we can get RHS=
\begin{align}
&\left(\frac{1}{2}\theta u^2 M_3\ex(k+1)+\frac{1}{2}\theta u^2 \ex(k+1)\ex(\sum_{n=1}^{k}\mathbb{M}_{2n})\bullet M_3 \nonumber \right.\\
&\left.+\frac{1}{4}u^4\big([\ex(k+1),M_5]+\ex(k+1)(\ex(\sum_{n=1}^{k}\mathbb{M}_{2n})-1)\bullet M_5\big)\right) \nonumber\\
&\ \ \cdot \prod_{n=1}^{k}\ex(\theta \bar{L}_{2n}u^{2n}). \nonumber
\end{align}
Cancel the ending terms of both sides and compare the coefficent of $u^{2k+2}$, we will get
\begin{align}
&\theta[\bar{L}_{2k+2},L_1]+C_{2k+2}\left(\big(\ex(\sum_{n=1}^{k}\mathbb{M}_{2n})-1\big)\bullet L_1\right)  \nonumber\\
&=C_{2k}\left(\frac{1}{2}\theta  M_3\right)+C_{2k}\left( \frac{1}{2}\theta \ex(\sum_{n=1}^{k}\mathbb{M}_{2n})\bullet M_3\right)+C_{2k-2}\left(\frac{1}{4}\big(\ex(\sum_{n=1}^{k}\mathbb{M}_{2n})-1\big)\bullet M_5\right),\nonumber
\end{align}
Hence we have $\theta [\bar{L}_{2k+2},L_1]=R_{2k+2}  M_{2k+3}$, where 
\begin{align}
R_{2k+2}=&C_{2k}\left(\frac{1}{2}\theta \right)+C_{2k}\left(\frac{1}{2}\theta \ex(\sum_{n=1}^{k}\theta R_{2n}u^{2n})\right)\nonumber\\
&+C_{2k-2}\left(\frac{1}{4}\big(\ex(\sum_{n=1}^{k}\theta R_{2n}u^{2n})-1\big)\right)-C_{2k+2}\left(\ex(\sum_{n=1}^{k}\theta R_{2n}u^{2n})\right).\nonumber
\end{align}
Then we need to prove $D_{2k}=0$.

Consider the formal power series: $\theta \sum_{n=0}^{\infty}D_{2n}u^{2n}$. 
\begin{align}
\theta \sum_{n=0}^{\infty}D_{2n}u^{2n}=&-\left[[\theta \sum_{n=0}^{\infty}\tilde{L}_{2n}u^{2n},\tilde{L}_{2N+1}],L_1\right]+\left[\big[\tilde{L}_{2N+2},\frac{1}{2}\theta u^2L_1 \theta \sum_{n=0}^{\infty}\tilde{L}_{2n}u^{2n}\nonumber\right.\\
&\left.+\frac{1}{2}\theta u^2  \theta \sum_{n=0}^{\infty}\tilde{L}_{2n}u^{2n}L_1+\frac{1}{4}u^4[\theta \sum_{n=0}^{\infty}\tilde{L}_{2n}u^{2n},M_3]\big] ,L_1\right]\nonumber\\
=&-[A,L_1]+\left[[\tilde{L}_{2N+2},B] ,L_1\right].\nonumber
\end{align}

Now we calculate $A$ and $B$.

The calculation is similar to the previous. We have 

$A=A_{2k+2}\prod_{n=1}^{k}\ex(\theta \bar{L}_{2n}u^{2n}),$

where 

$A_{2k+2}=[\ex(k+1),\tilde{L}_{2N+1}]+\ex(k+1)\left(\ex(\sum_{n=1}^{k}\mathbb{M}_{2n})-1\right)\bullet \tilde{L}_{2N+1}.$

$B=B_{2k+2}\prod_{n=1}^{k}\ex(\theta \bar{L}_{2n}u^{2n}),$

where
\begin{align}
B_{2k+2}=&\frac{1}{2}\theta u^2L_1\ex(k+1)+\frac{1}{2}\theta u^2\ex(k+1)\ex(\sum_{n=1}^{k}\mathbb{M}_{2n})\bullet L_1 \nonumber\\
&+\frac{1}{4}u^4\left([\ex(k+1),M_3]+\ex(k+1)\big(\ex(\sum_{n=1}^{k}\mathbb{M}_{2n})-1\big)\bullet M_3\right) \nonumber
\end{align}
Then 
\begin{align}
\theta \sum_{n=0}^{\infty}D_{2n}&=-\left[A_{2k+2}\prod_{n=1}^{k}\ex(\theta \bar{L}_{2n}u^{2n}),L_1\right]+\left[[\tilde{L}_{2N+2},B_{2k+2}\prod_{n=1}^{k}\ex(\theta \overline{L}_{2n}u^{2n})] ,L_1\right]\nonumber\\
&=-\left[A_{2k+2}\prod_{n=1}^{k}\ex(\theta \bar{L}_{2n}u^{2n}),L_1\right]+\left[[\tilde{L}_{2N+2},B_{2k+2}]\prod_{n=1}^{k}\ex(\theta \bar{L}_{2n}u^{2n}) ,L_1\right]\nonumber\\
&=\left[(-A_{2k+2}+[\tilde{L}_{2N+2},B_{2k+2}])\prod_{n=1}^{k}\ex(\theta \bar{L}_{2n}u^{2n}),L_1\right].\nonumber
\end{align}
We need to prove that the coefficient of $u^{2k}$ in $\theta \sum_{n=0}^{\infty}D_{2n}u^{2n}$ is 0.

It is enough to prove $C_{2k}\left([\Phi \prod_{n=1}^{k}\ex(\theta \bar{L}_{2n}u^{2n}),L_1]\right)=0$, where
\begin{align}
 \Phi=&-\ex(\sum_{n=1}^{k}\mathbb{M}_{2n})\bullet \tilde{L}_{2N+1}+\left[\tilde{L}_{2N+2},\frac{1}{2}\theta u^2 L_1 \nonumber\right.\\
&\left.+\frac{1}{2}\theta u^2 \ex(\sum_{n=1}^{k}\mathbb{M}_{2n})\bullet L_1+\frac{1}{4}u^4 \big(\ex(\sum_{n=1}^{k}\mathbb{M}_{2n})-1\big)\bullet M_3\right].\nonumber
\end{align}

 We claim that $\Phi_{2i}:=C_{2i}(\Phi)=0$ for $i\le k$.  

We have $\Phi_{2i}=\phi_{2i}\left(\frac{1}{2}\theta M_{2i+1}\tilde{L}_{2N}+\frac{1}{2}\theta \tilde{L}_{2N}M_{2i+1}+\frac{1}{4}[\tilde{L}_{2N-2},M_{2i+3}]\right)$, where 
\begin{align}
\phi_{2i}=&C_{2i}\Big(-\ex(\sum_{n=1}^{k}\theta R_{2n}u^{2n})+\frac{1}{2}\theta u^2+\frac{1}{2}\theta u^2\ex(\sum_{n=1}^{k}\theta R_{2n}u^{2n})\nonumber\\
&+\frac{1}{4}u^4(\ex(\sum_{n=1}^{k}\theta R_{2n}u^{2n})-1)\Big).\nonumber
\end{align}
Then $\Phi_{2i}=0$ for $i\le k$ holds by this and the induction hypethesis of $R_{2i}$ for $i\le k$.

It is not difficult to get the general formula of $R_{2m}$, we omit the details.

\end{proof}

\begin{corol} \label{corol:tilde}

We have the following commuting relations in $B(\mathbb{V})$.

(a)  $[\tilde{L}_{2m},\tilde{L}_{2n}]=0=[\bar{L}_{2m},\bar{L}_{2n}]$.

(b)  $[M_{2m+1},M_{2n+1}]=0 $ if $m+n$ is odd, $[M_{2m+1},M_{2n+1}]=(-1)^{\frac{m-n}{2}}2M_{m+n+1}^2 $ if $m+n$ is even.

(c)  $[\bar{L}_{2m},M_{2n+1}]=R_{2m}M_{2m+2n+1}$, where $R_{4m}=0$ and $R_{4m+2}=\frac{1}{2^{2m}}\frac{[2m+1]_q}{2m+1}$.

By these relations, the only possible imaginary root vectors are $L_{4n}$ at $4n \delta$, $L_{4n+2}$ and $M_{2n+1}^2$ at $(4n+2)\delta$, $M_{2n+1}$ at $(2n+1)\delta$. 

\end{corol}

\begin{proof}
These relations hold immediately by Prop.~\ref{propo:someimaginaryrootvectorscomm} and  Prop.~\ref{propo:central}. Then by induction on degree we can get these possible imaginary root vectors.
\end{proof}

\section{Root multiplicities of $B(\mathbb{V})$}

In this subsection, we will determine the multiplicities of all the roots and give a PBW basis of $B(\mathbb{V})$. The main tool is the subquotient $K_{\ge1}/K_{>1}$. In detail, the key step is finding the primitive element at certain degree in $K_{\ge 1}/K_{>1}$. At $U_{v}(\hat{sl_2})$ case, the Drinfeld generators defined by exponential are exactly these primitive elements.

\subsection{$(2n+1)\delta$}

By Corollary. ~\ref{corol:tilde}, we know the only possible root vector at $(2n+1)\delta$ is $M_{2n+1}$.

It is easy to get the following formulas in the subquotient.

$\Delta(L_2)=L_2 \otimes 1-\theta L_1\otimes L_1 +1\otimes L_2$

$\Delta(L_1)=L_1 \otimes 1+1\otimes L_1$,

$\Delta(L_1^2)=L_1^2 \otimes 1+1\otimes L_1^2$, 

$\Delta(M_3)=M_3\otimes 1 +2\theta L_1^2\otimes L_1-2\theta L_1\otimes L_1^2+1\otimes M_3$

$\Delta(M_5)=M_5\otimes 1-4\theta^2 L_1^3\otimes L_1^2+2\theta L_1^2\otimes M_3-2\theta M_3\otimes L_1^2 -4\theta^2 L_1^2\otimes L_1^3+1\otimes M_5$

$\cdots$

Let $\mathbb{O}_{2n+1}$ be the subalgebra generated by $\{M_{2i+1}, 0\le i\le n\}$.

In general, by an easy induction we can get 
\begin{align}
\Delta(M_{2n+1}) &\in M_{2n+1}\otimes 1+1\otimes M_{2n+1}+2\theta L_1^2\otimes M_{2n-1}-2\theta M_{2n-1}\otimes L_1^2 \nonumber\\
&+ \mathbb{O}_{2n-3}\otimes \mathbb{O}_{2n-3}.\nonumber
\end{align}
$M_{2n+1}$ is a root vector if and only if $M_{2n+1}$ is a linear combination of $L_1^{k_1}M_{3}^{k_2}M_{5}^{k_3}\cdots M_{2n-1}^{k_n}$.  The comultiplication of $M_{2n+1}$ has  term $2\theta L_1^2\otimes M_{2n-1}-2\theta M_{2n-1}\otimes L_1^2$ but the comultiplication of $L_1^2 M_{2n-1}$ can only have term
$L_1^2\otimes M_{2n-1}+ M_{2n-1}\otimes L_1^2$. Hence by induction we can see $M_{2n+1}$ must be a root vector when $q$ is not a root of unity.

\subsection{$(4n+2)\delta$ and $4n \delta$}

\subsubsection{$M_{2n+1}^2$}

Since $M_{2n+1}$ is a root vector and $\chi((2n+1)\delta, (2n+1)\delta)=-1$, we should consider $M_{2n+1}^2$.

Define $\overline{M_{2n+1}^{2}}$ by the following power series,

 $$\arc(2\theta \sum_{n=0}^{\infty}M_{2n+1}^2 u^{4n+2})=2\theta \sum_{n=0}^{\infty} \overline{M_{2n+1}^{2}}u^{4n+2}.$$

\begin{propo} $\overline{M_{2n+1}^{2}}$ is primitive in $K_{\ge 1}/K_{>1}$. The leading term of $[X_1, \overline{M_{2n+1}^{2}}]$ is $(-1)^n 2^{2n}\frac{[2n+1]_q}{2n+1}X_{4n+3}$, which means $\overline{M_{2n+1}^{2}}$ is not 0 in $B(\mathbb{V})$ when q is not a root of unity. \label{propo:2n+1square}
\end{propo}

\begin{proof}   We have $\arc(x)=\frac{i}{2}\lo(\frac{i+x}{i-x})$. Recall $X:=\sum_{n=0}^{\infty}2\theta M_{2n+1}^2 u^{4n+2}$. We have

 $$\lo(\frac{i+X}{i-X})=\frac{2}{i}\sum_{n=0}^{\infty}2\theta \overline{M_{2n+1}^2}u^{4n+2}.$$
Then it is equivalent to prove $\frac{i+X}{i-X}$ is group-like. $\frac{i+X}{i-X}=\frac{2}{1+iX}-1$, i.e. we want to prove $\Delta(\frac{2}{1+iX}-1)=(\frac{2}{1+iX}-1)\otimes (\frac{2}{1+iX}-1)$. Let $Y=\sqrt{2\theta}\sum_{n=0}^{\infty}M_{2n+1}u^{2n+1}i^n$. We have  $Y^2=X$  by Corollary ~\ref{mrelations}.

By induction it is not difficult to prove that
 $$\Delta(Y)=(Y\otimes 1+1\otimes Y)\frac{1}{1-iY\otimes Y},$$ in $K_{\ge 1}/K_{>1}$, we omit the detail. Then
\begin{align}
 \Delta(X)=\Delta(Y^2)&=(Y\otimes 1+1\otimes Y)\frac{1}{1-iY\otimes Y}\cdot (Y\otimes 1+1\otimes Y)\frac{1}{1-iY\otimes Y}\nonumber\\
&=(Y\otimes 1+1\otimes Y)^2\cdot \frac{1}{1+iY\otimes Y}\cdot \frac{1}{1-iY\otimes Y}\nonumber\\
&=(Y^2\otimes 1+1\otimes Y^2)\cdot \frac{1}{1-Y^2\otimes Y^2}\nonumber\\
&=(1\otimes X + X\otimes 1)\cdot \frac{1}{1-X\otimes X}. \nonumber
\end{align}

Then $\Delta(1+iX)=1+i(1\otimes X +X\otimes 1)\cdot \frac{1}{1-X\otimes X}=\frac{(1+iX)\otimes(1+iX)}{1-X\otimes X}$. Now $\Delta(\frac{2}{1+iX}-1)=(\frac{2}{1+iX}-1)\otimes (\frac{2}{1+iX}-1)$ holds immediately.

For the leading term of $[X_1, \overline{M_{2n+1}^{2}}]$, it is easy to get the leading term of $[X_1, 2\theta M_{2n+1}^2 u^{4n+2}]$ is $2\theta (-1)^n 2^{2n}X_{4n+3}u^{4n+2}$ by Prop. ~\ref{propo1}. Then the leading term of $[X_1, \overline{M_{2n+1}^{2}}]$ is

$$\frac{1}{2\theta}C_{4n+2}\left(\arc \big(2\theta \sum_{k=0}^{\infty}(-1)^k 2^{2k}u^{4k+2}\big)\right)X_{4n+3},$$
i.e
$$\frac{1}{2\theta}C_{4n+2}\left(\arc (2\theta u^2\frac{1}{1+4 u^4})\right)X_{4n+3}.$$
Then by some combinatoric techniques we can prove the leading term of $[X_1, \overline{M_{2n+1}^{2}}]$ is $(-1)^n 2^{2n}\frac{[2n+1]_q}{2n+1}X_{4n+3}$, we omit this detail.

\end{proof}

\begin{corol} $M_{2n+1}^2$ are root vectors when q is not a root of unity.\label{corol:mul}
\end{corol}

\begin{proof} Suppose $M_{2n+1}^2$ is not a root vector. Then the only possibility is $M_{2n+1}^2$ is a linear combination of products of $M_{2k+1}^2$ in $B(\mathbb{V})$, where $k\le n-1$. Equivalently, $\overline{M_{2n+1}^2}$ is a linear combination of products of $\overline{M_{2k+1}^2}$ where $k\le n-1$, denoted this linear combination by $\Pi$. Then $\overline{M_{2n+1}^2}-\Pi$ is 0 in $B(\mathbb{V})$ and must be primitive in $K_{\ge1}/K_{>1}.$  Then it is easy to see the only possibility is $\Pi=0$, since  $\overline{M_{2k+1}^2}$ are primitive in $K_{\ge 1}/K_{>1}$ and are root vectors in $B(\mathbb{V})$ for $k\le n-1$.  Then $\overline{M_{2n+1}^2}$ is 0 in $B(\mathbb{V})$, contradiction.
\end{proof}

\subsubsection{$L_{2n}$}

\begin{lemma} \label{coefflemma}
The leading term of $[X_1, \bar{L}_{2n}]$ is $\frac{(q^{2n}+(-1)^n)(q^n+(-1)^{n+1})}{nq^{n-1}(q^2-1)}X_{2n+1}$.
\end{lemma}
\begin{proof}
We have $\lo(1+\theta \sum_{n=1}^{\infty}\tilde{L}_{2n}u^{2n})={\theta \sum_{n=1}^{\infty}\bar{L}_{2n}u^{2n}}.$ By Prop. ~\ref{propo:realandim}, we get the leading term of $[X_1, \theta \tilde{L}_{2n}u^{2n}]$ is 
\begin{align}
&\theta \left(q^{2n-2}\frac{1-(-q^{-2})^n}{1+q^{-2}}+q^{2n-1}\right)X_{2n+1}u^{2n} \nonumber\\
&=\frac{\theta}{1+q^{-2}}\sum_{n=1}^{\infty}\left((q^{-1}+q^{-2}+q^{-3})(q^2 u^2)^n-q^{-2}(-u^2)^n\right)X_{2n+1}. \nonumber
\end{align}
Then we have the leading term of $[X_1, \bar{L}_{2n}]$ is
$$\theta^{-1}C_{2n}\left(\lo \Big(1+\frac{\theta}{1+q^{-2}}\sum_{n=1}^{\infty}\big((q^{-1}+q^{-2}+q^{-3})(q^2 u^2)^n-q^{-2}(-u^2)^n\big)\Big)\right)X_{2n+1},$$
i.e
$$\theta^{-1}C_{2n}\left(\lo \Big(1+\frac{\theta}{1+q^{-2}}\big((q^{-1}+q^{-2}+q^{-3})q^2 u^2\frac{1}{1-q^2 u^2}-q^{-2}(-u^2)\frac{1}{1+u^2} \big)\Big)\right)X_{2n+1}.$$
Then by some combinatoric techniques we have the leading term of $[X_1, \bar{L}_{2n}]$ is $\frac{(q^{2n}+(-1)^n)(q^n+(-1)^{n+1})}{nq^{n-1}(q^2-1)}X_{2n+1}$ and we omit this detail.

\end{proof}

Define $\mathring{L}_{4n}$ by the following, recall  $\alpha=\sqrt{\frac{iq}{2}}$, $\beta=\sqrt{\frac{iq^{-1}}{2}}$,

\begin{align}
\sum_{n=1}^{\infty}2\theta \mathring{L}_{4n}u^{4n}=&\sum_{n=1}^{\infty}2\theta \bar{L}_{4n}u^{4n} \nonumber\\
&+\frac{1}{2}log\left(1+\left(\sum_{n=0}^{\infty}2\theta M_{2n+1}^2 (\alpha u )^{4n+2}\right)^2\right)\nonumber\\
&+\frac{1}{2}log\left(1+\left(\sum_{n=0}^{\infty}2\theta M_{2n+1}^2 (\beta u )^{4n+2}\right)^2\right).\nonumber
\end{align}

\begin{propo} $\mathring{L}_{4n}$ is primitive in the subquotient $K_{\ge 1}/K_{>1}$. The leading term of $[X_1, \mathring{L}_{4n}]$ is $\frac{[4n]_q}{4n}X_{4n+1}$, which means $\mathring{L}_{4n}$ is not 0 in $B(\mathbb{V})$ when $q$ is not a root of unity. \label{propo:4n}
\end{propo}

\begin{proof}
Recall $\tilde{L}(u):=\sum_{n=0}^{\infty}\theta \tilde{L}_{2n}u^{2n}$,  let $\bar{L}(u):=\sum_{n=1}^{\infty}\theta \bar{L}_{2n}u^{2n}$, i.e. we have $\tilde{L}(u)=\ex(\bar{L}(u))$.

By Corollary ~\ref{corol:comulformula}, in the subquotient, we have\\

 $\Delta(\tilde{L}_{2n})=$

$$\tilde{L}_{2n}\otimes 1-\theta \tilde{L}_{2n-1}\otimes \tilde{L}_{1}+\theta \tilde{L}_{2n-2}\otimes \tilde{L}_{2}-\cdot\cdot\cdot +\theta \tilde{L}_{2}\otimes \tilde{L}_{2n-2}-\theta \tilde{L}_{1}\otimes \tilde{L}_{2n-1}+1\otimes \tilde{L}_{2n}
$$ 

$$+\theta \frac{1}{2}\hat{L}_{2n-2}\otimes \frac{1}{2}\hat{L}_{2}+\theta \frac{1}{2}\hat{L}_{2n-3}\otimes \frac{1}{2}\hat{L}_{3}+\cdot\cdot\cdot +\theta \frac{1}{2}\hat{L}_{3}\otimes \frac{1}{2}\hat{L}_{2n-3}+\theta \frac{1}{2}\hat{L}_{2}\otimes \frac{1}{2}\hat{L}_{2n-2}$$\\
i.e.\\

$\Delta(\tilde{L}(u))=$

$$\sum_{n=0}^{\infty}\theta \tilde{L}_{2n}u^{2n}\otimes \sum_{n=0}^{\infty}\theta \tilde{L}_{2n}u^{2n}-\sum_{n=0}^{\infty}\theta \tilde{L}_{2n+1}u^{2n+1}\otimes \sum_{n=0}^{\infty}\theta \tilde{L}_{2n+1}u^{2n+1}$$

$$+\sum_{n=1}^{\infty}\frac{1}{2}\theta \hat{L}_{2n}u^{2n}\otimes \sum_{n=1}^{\infty}\frac{1}{2}\theta \hat{L}_{2n}u^{2n}+\sum_{n=1}^{\infty}\frac{1}{2}\theta \hat{L}_{2n+1}u^{2n+1}\otimes \sum_{n=1}^{\infty}\frac{1}{2}\theta \hat{L}_{2n+1}u^{2n+1}$$

$:=*(u)$.

We have $\Delta(\tilde{L}(u))=*(u)$, $\Delta(\tilde{L}(iu))=*(iu)$, so $\Delta(\sum_{n=1}^{\infty}2\theta \bar{L}_{4n}u^{4n})=\lo(*(u)*(iu))$.

By Prop. ~\ref{seriesrelations} (a), (b) and (e) it is easy to get 
\begin{align}
*(u)=&\left(1-A\otimes A+\frac{1}{2}B\otimes \frac{1}{2}B \nonumber \right.\\
&\left.+\big(\theta^{-1}(\frac{1}{4}B^2+A^2)+A\cdot \frac{1}{2}B\big)\otimes \big(\theta^{-1}(\frac{1}{4}B^2+A^2)+A\cdot \frac{1}{2}B\big)\right)\nonumber\\
&\cdot \ex\left(\bar{L}(u)\otimes 1+1\otimes \bar{L}(u)\right)\nonumber\\
&:=\#(u)\cdot \ex\left(\bar{L}(u)\otimes 1+1\otimes \bar{L}(u)\right).\nonumber
\end{align}
and 
\begin{align}
*(u)=&\ex \left(\bar{L}(u)\otimes 1+1\otimes \bar{L}(u)\right)\cdot \left(1-A'\otimes A'+\frac{1}{2}B'\otimes \frac{1}{2}B' \nonumber  \right.\\
&\left.+\big(\theta^{-1}(\frac{1}{4}B'^2+A'^2)+A'\cdot \frac{1}{2}B'\big)\otimes \big(\theta^{-1}(\frac{1}{4}B'^2+A'^2)+A'\cdot \frac{1}{2}B'\big)\right) \nonumber\\
&:=\ex \left(\bar{L}(u)\otimes 1+1\otimes \bar{L}(u)\right)\cdot \#'(u).\nonumber
\end{align}
So 
\begin{align}
*(u)\cdot *(iu)&=\#(u)\cdot \ex \left(\bar{L}(u)\otimes 1+1\otimes \bar{L}(u)\right)\cdot \ex \left(\bar{L}(iu)\otimes 1+1\otimes \bar{L}(iu)\right)\cdot \#'(iu) \nonumber\\
&=\#(u)\cdot \ex \left(\sum_{n=1}^{\infty}2\theta \bar{L}_{4n}u^{4n}\otimes 1+1\otimes \sum_{n=1}^{\infty}2\theta \bar{L}_{4n}u^{4n}\right)\cdot \#'(iu), \nonumber
\end{align}
then 
$$\lo(*(u)*(iu))=\sum_{n=1}^{\infty}2\theta \bar{L}_{4n}u^{4n}\otimes 1 + 1\otimes \sum_{n=1}^{\infty}2\theta \bar{L}_{4n}u^{4n}+\lo(\#(u)\#'(iu)),$$ since $\bar{L}_{4n}$ commute with all $M_{2m+1}$ by Corollary ~\ref{corol:tilde}.                   

Recall that $A'(iu)=iA(u)$, $B'(iu)=i^{-1}B(u)$. Then $\lo (\#(u)\#'(iu))$ is all about A and B and their products. 

We have 
\begin{align}
 \#(u)\#'(iu)=&\left(1-A\otimes A+\frac{1}{2}B\otimes \frac{1}{2}B+\big(\theta^{-1}(\frac{1}{4}B^2+A^2)+A\cdot \frac{1}{2}B\big)\nonumber\right.\\
&\left.\otimes \big(\theta^{-1}(\frac{1}{4}B^2+A^2)+A\cdot \frac{1}{2}B\big)\right)\nonumber\\
&\cdot \left(1+A\otimes A-\frac{1}{2}B\otimes \frac{1}{2}B+\big(-\theta^{-1}(\frac{1}{4}B^2+A^2)+A\cdot \frac{1}{2}B\big)\nonumber\right.\\
&\left.\otimes \big(-\theta^{-1}(\frac{1}{4}B^2+A^2)+A\cdot \frac{1}{2}B\big)\right).\nonumber
\end{align}
Omit a long but purely elementry calculation using Prop. ~\ref{seriesrelations} (c), (d), we get $$\#(u)\#'(iu)=\left(1-X(\alpha u)\otimes X(\alpha u)\right)\cdot\left(1-X(\beta u)\otimes X(\beta u)\right).$$

Recall that by definition we have 
\begin{align}
&2\cdot \sum_{n=1}^{\infty}2\theta \mathring{L}_{4n}u^{4n}\nonumber\\
&=2\cdot(\sum_{n=1}^{\infty}2\theta \bar{L}_{4n}u^{4n})+\lo \left(1+X^2 (\alpha u)\right)+\lo \left(1+X^2 (\beta u)\right).\nonumber
\end{align}
We have known $\Delta(X)=(X\otimes 1+1\otimes X)\cdot \frac{1}{1-X\otimes X}$, then $\Delta(1+X^2)=\frac{(1+X^2\otimes 1)(1+1\otimes X^2)}{(1-X\otimes X)^2}$.

Now 
\begin{align}
\Delta(2\cdot \sum_{n=1}^{\infty}2\theta \mathring{L}_{4n}u^{4n})=&2\sum_{n=1}^{\infty}2\theta \bar{L}_{4n}u^{4n}\otimes 1 + 1\otimes 2\sum_{n=1}^{\infty}2\theta \bar{L}_{4n}u^{4n}+2\lo(\#(u)\#'(iu))\nonumber\\
&+\lo \left(\frac{(1+X^2(\alpha u)\otimes 1)(1+1\otimes X^2(\alpha u))}{(1-X(\alpha u)\otimes X(\alpha u))^2}\right)\nonumber\\
&+\lo \left(\frac{(1+X^2(\beta u)\otimes 1)(1+1\otimes X^2(\beta u))}{(1-X(\beta u)\otimes X(\beta u))^2}\right)\nonumber\\
=&\left(2\sum_{n=1}^{\infty}2\theta \bar{L}_{4n}u^{4n}+\lo \left(1+X^2 \left(\alpha u\right)\right)+\lo \left(1+X^2 \left(\beta u\right)\right)\right)\otimes 1 \nonumber\\
&+1\otimes \left(2\sum_{n=1}^{\infty}2\theta \bar{L}_{4n}u^{4n}+\lo \left(1+X^2 \left(\alpha u\right)\right)+\lo \left(1+X^2 \left(\beta u\right)\right)\right), \nonumber
\end{align}
which means $\mathring{L}_{4n}$ is primitive in our subquotient.

For the leading term of $[X_1, \mathring{L}_{4n}]$, by Lemma ~\ref{coefflemma} we have the leading term of $[X_1, \bar{L}_{2n}]$ is $\frac{(q^{2n}+(-1)^n)(q^n+(-1)^{n+1})}{nq^{n-1}(q^2-1)}X_{2n+1}$. Recall that the leading term of $[X_1, M_{2n+1}^2]$ is $(-1)^n 2^{2n}X_{4n+3}$. Then we have the leading term of $[X_1, \mathring{L}_{4n}]$ is
\begin{align}
&\frac{1}{2\theta}\left(2\theta \frac{(q^{4n}+1)(q^{2n}-1)}{2n q^{2n-1}(q^2-1)}+C_{4n}\Big(\frac{1}{2}\lo \big(1+(2\theta(\alpha u)^2\frac{1}{1+4\alpha^4u^4})^2\big)\Big)\nonumber \right.\\
&\left.+C_{4n}\Big(\frac{1}{2}\lo \big(1+(2\theta(\beta u)^2\frac{1}{1+4\beta^4u^4})^2\big)\Big)\right)X_{4n+1}\nonumber
\end{align}
Then by some combinatoric techniques we can prove the leading term of $[X_1, \mathring{L}_{4n}]$ is $\frac{[4n]_q}{4n}X_{4n+1}$, we omit this detail.
\end{proof}

\begin{propo} \label{propo:applyBCHDformula} Let $\mathbb{O}$ be the subalgebra in $B(\mathbb{V})$ generated by $M_{2n-1}$ for $n\in \mathbb{N}$. Then in the subquotient $K_{\ge 1}/K_{>1}$, $\Delta(\bar{L}_{2n})\in \bar{L}_{2n}\otimes 1+1\otimes \bar{L}_{2n}+\mathbb{O}\otimes \mathbb{O}$.
\end{propo}

\begin{proof}
Following the notations in the last proposition, we have 
\begin{align}
\Delta(\bar{L}(u))&=\lo \left(\Delta(\tilde{L}(u))\right)=\lo(*(u)) \nonumber\\
&=\lo\left(\#(u)\cdot \ex(\bar{L}(u)\otimes 1+1\otimes \bar{L}(u))\right).\nonumber\\
&=\lo\Bigg(\ex\bigg(\lo(\#(u))\bigg)\cdot \ex \bigg(\bar{L}(u)\otimes 1+1\otimes \bar{L}(u)\bigg)\Bigg)
\end{align}
To deal with this type of formula, we have the famous Baker-Campbell-Hausdorff formula, see Chapter 3 of \cite{BCHDformulabook}. By BCH formula and Corollary ~\ref{corol:tilde} (c), this proposition holds immediately.
\end{proof}

\begin{corol}
 $L_{2n}$ are root vectors for $n\ge 1$ when q is not a root of unity.
\end{corol}

\begin{proof} 
We prove by induction on $n$. 

$L_2$ is a root vector because $[L_2, L_1]=M_3$ and $M_3$ is a root vector. 

For $L_4$, suppose $L_4$ is not a root vector, then $L_4=\sum k_{i_1,i_2,j_1}L_1^{i_1}M_3^{i_2}L_2^{j_1}$, or equivalently $\bar{L}_4=\sum k'_{i_1,i_2,j_1}L_1^{i_1}M_3^{i_2}\bar{L}_2^{j_1}$. By Prop. ~\ref{propo:applyBCHDformula} and $L_2$ is a root vector, $\bar{L}_2^2$ and $\bar{L}_1^2\bar{L}_2$ cannot appear in this linear combination. Then $\bar{L}_4$ can only be a linear combination of $L_1 M_3$ and $L_1^4$. Then use $[\bar{L}_2,\ \ \ ]$ act on $\bar{L}_4$, we get $M_3^2+L_1M_5=0$, contradict to $M_{2n-1}$ are root vectors for $n\in \mathbb{N}$. Then $\bar{L}_4$ can only equal to $k L_1^4$, equivalently $\mathring L_4=k (L_1^2)^2$. Since $\mathring L_4$ and $L_1^2$ are primitive in $K_{\ge 1}/K_{>1}$, then the only possibility is $\mathring L_4=0$, but by Prop.~\ref{propo:4n} $\mathring L_{4}$ is not 0. Then $L_{4}$ must be a root vector.

Suppose it holds for $n$, we consider $n+1$. 

If $n$ is even, then $2(n+1)=4m+2$ for $m=\frac{n}{2}$. Suppose $L_{4m+2}$ is not a root vector, then $$L_{4m+2}=\sum k_{i_1,i_2,...,i_{2m+1},j_1,...,j_{2m}}L_1^{i_1}M_3^{i_2}\cdots M_{4m+1}^{i_{2m+1}}L_2^{j_1}\cdots L_{4m}^{j_{2m}},$$ or equivalently $$\bar L_{4m+2}=\sum k'_{i_1,i_2,...,i_{2m+1},j_1,...,j_{2m}} L_1^{i_1}M_3^{i_2}\cdots M_{4m+1}^{i_{2m+1}}\bar{L}_2^{j_1}\cdots \bar{L}_{4m}^{j_{2m}}.$$ By Prop. ~\ref{propo:applyBCHDformula} and induction hypothesis, $\bar{L}_2,...,\bar{L}_{4m}$ cannnot appear. Then $$\bar L_{4m+2}=\sum k''_{i_1,i_2,...,i_{2m+1}} L_1^{i_1}M_3^{i_2}\cdots M_{4m+1}^{i_{2m+1}}.$$
By Corollary ~\ref{corol:tilde} (c), $[\bar{L}_{4m+2},L_{1}]=R_{4m+2}M_{4m+3}$. When q is not a root of unity, we know $R_{4m+2}$ is not 0 and $M_{4m+3}$ is a root vector, contradiction. Then $L_{4m+2}$ must be a root vector.

If $n$ is odd, then $2(n+1)=4m$ for $m=\frac{n+1}{2}$. Suppose $L_{4m}$ is not a root vector, then $$L_{4m}=\sum k_{i_1,i_2,...,i_{2m},j_1,...,j_{2m-1}}L_1^{i_1}M_3^{i_2}\cdots M_{4m-1}^{i_{2m}}L_2^{j_1}\cdots L_{4m-2}^{j_{2m-1}}$$ or equivalently $$\mathring L_{4m}=\sum k'_{i_1,i_2,...,i_{2m},j_1,...,j_{2m-1}} L_1^{i_1}M_3^{i_2}\cdots M_{4m+1}^{i_{2m}}\bar{L}_2^{j_1}\cdots \bar{L}_{4m-2}^{j_{2m-1}}.$$ For the same reason as last paragraph, $\bar{L}_2,...,\bar{L}_{4m-2}$ cannnot appear. Use $[\bar L_2,\ \ \ ]$ act on $\mathring L_{4m}$, similar to  the case of $L_4$ we get $$\mathring L_{4m}=\sum k''_{i_1,i_2,...,i_{2m}} L_1^{i_1}M_3^{i_2}\cdots M_{4m+1}^{i_{2m}},$$ where $i_1,i_2,...,i_{2m}$ must be even. Then equivalently $$\mathring L_{4m}=\sum k'''_{i_1,i_2,...,i_{2m}} (L_1^2)^{\frac{i_1}{2}}(\overline{M_3^2})^{\frac{i_2}{2}}\cdots (\overline{M_{4m+1}^2})^{\frac{i_{2m}}{2}}.$$ Since $\mathring L_{4m},L_1^2,\overline{M_3^2},...,\overline{M_{4m+1}^2}$ are primitive in $K_{\ge 1}/K_{>1}$, the only possibility is $\mathring L_{4m}=0$, but by Prop.~\ref{propo:4n} $\mathring L_{4m}$ is not 0. Then $L_{4m}$ must be a root vector.

\end{proof}

Now we have the following theorem immediately.

\begin{theor}
 Suppose $q$ is not a root of unity, the multiplicities of $4n \delta$ and $(2n+1)\delta $ are 1, of $(4n+2)\delta$ is 2, which coincide with the root multiplicities of the Lie superalgebra $A(0,2)^{(4)}$. The set of decreasing ordered products of elements in $\{X_n, Y_n, L_{2n},M_{2n-1},n\in \mathbb{N}\}$ form a PBW basis of $B(\mathbb{V})$. The order in $\{X_n, Y_n, L_{2n},M_{2n-1},n\in \mathbb{N}\}$ is \begin{align}
X_1&<X_2<\cdots<X_{n}<\cdots \nonumber \\
&<\cdots<L_{2n}<L_{2n-2}<L_{2n-4}<\cdots<L_2 \nonumber \\
&<\cdots<M_{2n-1}<M_{2n-3}<\cdots<M_5<M_3<M_1 \nonumber \\
&<\cdots<Y_n<\cdots<Y_2<Y_1. \nonumber 
\end{align}

\end{theor}

\begin{remar} 
In Drinfeld second realization of $U_v(\hat{sl_2})$, there are $\bar{L}_n$ and $\sum_{n=0}^{\infty}\theta L_n=exp(\sum_{n=1}^{\infty}\theta \bar{L}_n)$. These $\bar{L}_n$ are primitive in $K_{\ge 1}/K_{>1}$ and satisfy $[X_n, \bar{L}_m]=\frac{[m]_q}{m}X_{n+m}$ in $U_v(\hat{sl_2})$, they are called Drinfeld generators.

Similarly, in Prop.~\ref{propo:2n+1square} and ~\ref{propo:4n} we get the leading term of $[X_1, \overline{M_{2n+1}^{2}}]$ is $(-1)^n 2^{2n}\frac{[2n+1]_q}{2n+1}X_{4n+3}$ and the leading term of $[X_1, \mathring{L}_{4n}]$ is $\frac{[4n]_q}{4n}X_{4n+1}$. In fact according to specific calculation in low degrees, $[X_1, \overline{M_{2n+1}^{2}}]=(-1)^n 2^{2n}\frac{[2n+1]_q}{2n+1}X_{4n+3}$ and $[X_1, \mathring{L}_{4n}]=\frac{[4n]_q}{4n}X_{4n+1}$ should hold but the author has not yet obtained a proof. If these hold, then $[X_k, \overline{M_{2n+1}^{2}}]$ = $(-1)^n 2^{2n}\frac{[2n+1]_q}{2n+1}X_{4n+2+k}$ and $[X_k, \mathring{L}_{4n}]$ = $\frac{[4n]_q}{4n}X_{4n+k}$, $\forall k\ge1$.

However here we can not get a full analog of Drinfeld second realization because for $M_{2n+1}$ and $\bar{L}_{4n+2}$ (non-central elements in the subquotient) it is impossible to get this kind of generators. Probably the relations between real root vectors and $M_{2n+1}$, $\bar{L}_{4n+2}$ can be described nicely in form of series.
\end{remar}

\section{The isomorphism of $B(\mathbb{V})$ and $U_{v}(A(0,2)^{(4)})^{+}$}

\begin{theor}

$B(\mathbb{V})$ is generated by Chevally generators $X_1$ and $Y_1$ with quantum Serre relations when $q$ is not a root of unity, which implies that $B(\mathbb{V})$ and $U_{v}(A(0,2)^{(4)})^{+}$ are isomorphic, where $q=v^2$. 

\end{theor}
\begin{proof}
According to previous sections, we have determined all the root vectors of  $B(\mathbb{V})$. It is equivalent to prove that all the commuting relation between these root vectors come from quantum Serre relations. It is enough to prove that the relations in Prop.~\ref{propo:realandim} can be derived by quantum Serre relations. We will only prove that first half  because the other half is completely similar. We omit some tedious details. The main difficulty is to prove the expression of $[X_2,L_{2n}]$  for $n>1$ in Prop.~\ref{propo:realandim} can be derived by quantum Serre relations. Relations in Prop.~\ref{propo:central} are all come from quantum Serre relations, and we know $2L_{2n+1}'-\theta L_{2n}'L_1=\theta L_1 L_{2n}+[L_2, L_{2n-1}']$. We use $[X_1, \quad \quad]$ act on this equation, we will get the result of $(q+q^{-1})[X_2, L_{2n}]$ and since q is not a root of unity, we will get the expression of $[X_2, L_{2n}]$ in terms of ordered products of root vectors by Prop.~\ref{propo:central}. On the other hand,  also by Prop.~\ref{propo:central}, it is enough to use 
\begin{align}
[X_2,L_{2n}]=&q^{2n-1}X_{2n+2}+q^{2n-2}\theta L_1'X_{2n+1}+q^{2n-3}\theta L_2X_{2n}\nonumber\\
&+\cdots+q\theta L_{2n-2}X_4+\theta L_{2n-1}'X_3\nonumber
\end{align}to get the expression of $[X_2,L_{2n}]$ in terms of  ordered products of root vectors. Then these two expressions must be the same. This means 
the expression of $[X_2,L_{2n}]$ in Prop.~\ref{propo:realandim} also come from quantum Serre relations. 
\end{proof}

\section*{Acknowledgements}

This paper was written during the visit of the author to University of Marburg. The author thanks his adviser  Prof. Istvan Heckenberger for many helpful discussions. The author thanks  his adviser Prof. Gongxiang Liu in Nanjing University for guideness on Kac-Moody algebras and Nichols algebras. The author also thanks  Prof. Hiroyuki Yamane for pointing out the main difficulty on $U_{v}(A(0,2)^{(4)})$ is imaginary root vectors do not necessarily commute with each other.


\providecommand{\bysame}{\leavevmode\hbox to3em{\hrulefill}\thinspace}
\providecommand{\MR}{\relax\ifhmode\unskip\space\fi MR }
\providecommand{\MRhref}[2]{%
  \href{http://www.ams.org/mathscinet-getitem?mr=#1}{#2}
}
\providecommand{\href}[2]{#2}

\end{document}